\documentclass[a4paper,UKenglish,cleveref, autoref, thm-restate]{lipics-v2021}
\title{Dichotomies for Maximum Matching Cut: $H$-Freeness, Bounded Diameter, Bounded Radius\footnote{An extended abstract of this paper appeared in the proceedings of MFCS 2023~\cite{LPR23b}.}}
\titlerunning{Dichotomies for Maximum Matching Cut} 

\author{Felicia Lucke}{Department of Informatics, University of Fribourg, Fribourg, Switzerland}{felicia.lucke@unifr.ch}{https://orcid.org/0000-0002-9860-2928}{}
\author{Daniël Paulusma}{Department of Computer Science, Durham University, Durham, UK}{daniel.paulusma@durham.ac.uk}{https://orcid.org/0000-0001-5945-9287}{}
\author{Bernard Ries$^*$}{Department of Informatics, University of Fribourg, Fribourg, Switzerland}{bernard.ries@unifr.ch}{https://orcid.org/0000-0003-4395-5547}{}

\authorrunning{F. Lucke and D. Paulusma and B. Ries} 
\Copyright{Felicia Lucke and Dani\"el Paulusma and Bernard Ries}

\nolinenumbers 
\hideLIPIcs

\usepackage{boxedminipage,tikz}
\usepackage{subcaption}
\usetikzlibrary{decorations.pathreplacing, calligraphy}
\usetikzlibrary{decorations.pathmorphing}
\usetikzlibrary{arrows,automata}
\usetikzlibrary{arrows.meta}
\usetikzlibrary{shapes.geometric}
\definecolor{nicered}{RGB}{204,0,0}
\definecolor{lightblue}{RGB}{153,204,255}
\definecolor{nicegreen}{RGB}{0,153,0}
 \tikzstyle{vertex}=[thin,circle,inner sep=0.cm, minimum size=1.7mm, fill=black, draw=black]
 \tikzstyle{svertex}=[thin,circle,inner sep=0.cm, minimum size=1.3mm, fill=black, draw=black]
 \tikzstyle{bvertex}=[thin,circle,inner sep=0.cm, minimum size=1.7mm, fill=lightblue, draw=lightblue]
 \tikzstyle{rvertex}=[thin,circle,inner sep=0.cm, minimum size=1.7mm, fill=nicered,draw=nicered]
 \tikzstyle{evertex}=[thin,circle,inner sep=0.cm, minimum size=1.7mm, fill=none,draw=black]
 \tikzstyle{hedge}=[thick, draw = gray]
 \tikzstyle{edge}=[thick, draw = gray]
 \tikzstyle{tedge}=[ultra thick, draw = black]
 \tikzstyle{tredge}=[ultra thick, draw = nicered]
 \tikzstyle{pedge}=[ultra thick, draw=lightblue]
 \tikzstyle{rededge}=[thick, draw = nicered]
 \tikzstyle{bluedge}=[thick, draw = lightblue]
 \tikzstyle{grnedge}=[thick, draw = nicegreen]
 \tikzstyle{edge}=[thick, draw = gray]
 \tikzstyle{br} = [decorate, ultra thick, decoration = {calligraphic brace}]
 \tikzstyle{wiggly} = [decorate, decoration = snake, thick, draw = gray,]

\newcommand{\NP}{{\sf NP}}
\newcommand{\PP}{{\sf P}}

\newcommand{\XP}{{\sf XP}}

\newcommand{\ssi}{\subseteq_i}
\newcommand{\si}{\supseteq_i}

\newcommand{\radius}{{\sf radius}}
\newcommand{\diam}{{\sf diameter}}

\newcommand{\mmc}{{\sc Maximum Matching Cut}}

\newcommand{\mdpm}{{\sc Maximum Disconnected Perfect Matching}}
\newcommand{\set}[1]{\ensuremath{ \left\lbrace #1 \right\rbrace }}

 \newtheorem{claim1}{Claim}[theorem]
 \newtheorem{open}{Open Problem} 
 
\begin{document}

\maketitle

\begin{abstract}
The {\sc (Perfect) Matching Cut} problem is to decide if a graph~$G$ has a (perfect) matching cut, i.e., a (perfect) matching that is also an edge cut of $G$. Both {\sc Matching Cut} and {\sc Perfect Matching Cut} are known to be \NP-complete. A perfect matching cut is also a matching cut with maximum number of edges. To increase our understanding of the relationship between the two problems, we perform a complexity study
for the {\sc Maximum Matching Cut} problem, which is
to determine a largest matching cut in a graph. Our results yield full dichotomies of {\sc Maximum Matching Cut} for graphs of bounded diameter, bounded radius and $H$-free graphs. A disconnected perfect matching of a graph $G$ is a perfect matching that contains a matching cut of~$G$. 
We also show how our new techniques can be used for finding a disconnected perfect matching 
with a largest matching cut 
for special graph classes. In this way we can prove that the 
decision problem {\sc Disconnected Perfect Matching} is polynomial-time solvable for $(P_6+sP_2)$-free graphs for every $s\geq 0$, 
extending a known result for $P_5$-free graphs (Bouquet and Picouleau, 2020).
\end{abstract}

\section{Introduction}\label{s-intro}

A {\it matching} $M$ (i.e., a set of pairwise disjoint edges) of a connected graph $G=(V,E)$ is a {\it matching cut} if $V$ can be partitioned into a set of blue vertices $B$ and a set of red vertices $R$, such that $M$ consists of all the edges with one end-vertex in $B$ and the other one in $R$. Graphs with matching cuts were introduced in 1970 by Graham~\cite{Gr70} (as {\it decomposable} graphs) to solve a problem on cube numbering. Other relevant applications include ILFI networks~\cite{FP82}, WDM networks~\cite{ACGH12}, graph drawing~\cite{PP01} and surjective graph homomorphisms~\cite{GPS12}.

The decision problem is called {\sc Matching Cut}: does a given connected graph have a matching cut? In 1984, Chv\'atal~\cite{Ch84} proved that it is \NP-complete even for graphs of maximum degree at most~$4$. Afterwards, parameterized and exact algorithms were given~\cite{AKK22,CHLLP21,GKKL22,GS21,KKL20,KL16}. A variant called {\sc Disconnected Perfect Matching}
``does a connected graph have a perfect matching that contains a matching cut?'' has also been studied~\cite{BP,FLPR23,LPR23a} (see Section~\ref{s-results} for more on this problem). Moreover, {\sc Matching Cut} was generalized, for every $d\geq 1$, to $d$-{\sc Cut} ``does a connected graph have an edge cut where each vertex has at most $d$ neighbours across the cut?''~\cite{AS21,GS21}. In particular, many results have appeared where the input for {\sc Matching Cut} was restricted to some special graph class, and this is what we do in our paper as well. 
We first discuss related work, restricting ourselves mainly to those classes relevant to our paper (see, for example,~\cite{CHLLP21} for a more comprehensive overview):\\[-10pt]
\begin{itemize}
\item graphs of bounded diameter;
\item graphs of bounded radius;
\item hereditary graph classes; in particular $H$-free graphs.
\end{itemize}

\noindent
The {\it distance} between two vertices $u$ and $v$ in a connected graph~$G$ is the {\it length} (number of edges) of a shortest path between $u$ and $v$ in~$G$. The {\it eccentricity} of a vertex $u$ is the maximum distance between $u$ and any other vertex of $G$. The {\it diameter}, denoted by $\diam(G)$, and {\it radius}, denoted by $\radius(G)$, are the maximum and minimum eccentricity, respectively, over all vertices of~$G$; note that $\radius(G)\leq \diam(G)\leq 2\cdot\radius(G)$ for every graph~$G$.

The {\sc Matching Cut} problem is polynomial-time solvable for graphs of diameter at most~$2$~\cite{BJ08,LL19}. This result was extended to graphs of radius at most~$2$~\cite{LPR22}. In contrast, the problem is \NP-complete for graphs of diameter at most~$3$~\cite{LL19}, yielding two dichotomies:
       
\begin{theorem}[\cite{LL19,LPR22}]\label{t-diameter}
For an integer $d\geq 1$, {\sc Matching Cut} for graphs of diameter~$d$ and for graphs of radius~$d$ is polynomial-time solvable if $d\leq 2$ and \NP-complete if $d\geq 3$.
\end{theorem}

\noindent
A class of graphs is {\it hereditary} if it is closed under vertex deletion. Hereditary graph classes include many well-known classes, such as those that are $H$-free for some graph~$H$.
A graph~$ G$ is {\it $H$-free} if $G$ does not contain $H$ as an {\it induced} subgraph, that is, $G$ cannot be modified into $H$ by a sequence of vertex deletions. For a set of graphs ${\cal H}$,
a graph $G$ is {\it ${\cal H}$-free} if $G$ is $H$-free for every $H\in {\cal H}$. If ${\cal H}=\{H_1,\ldots,H_p\}$ for some $p\geq 1$, we also say that $G$ is {\it $(H_1,\ldots,H_p)$-free}. Note that a class of graphs~${\cal G}$ is hereditary if and only if there is a set of graphs ${\cal H}$, such that every graph in ${\cal G}$ is ${\cal H}$-free. Hence, for a {\it systematic} complexity study, it is natural to first focus on the case where ${\cal H}$ has size~$1$; see, e.g.,~\cite{Ch12,CS05,DJP19,GJPS17,HMLW11,RS04}.

For an integer~$r\geq 1$, let $P_r$ denote the path on $r$ vertices, $K_{1,r}$ the star on $r+1$ vertices, and $K_{1,r}+e$ the graph obtained from $K_{1,r}$ by adding one edge (between two leaves). The graph $K_{1,3}$ is also known as the {\it claw}.
For $s\geq 3$, let $C_s$ denote the cycle on $s$ vertices. Let $H^*_1$ be the graph that looks like the letter ``$H$'', and for $i\geq 2$, let $H_i^*$ be the graph obtained from $H_1^*$ by subdividing the middle edge of $H_1^*$ exactly $i-1$ times; see also Figure~\ref{fig-Hstar}. We denote the {\it disjoint union} of two graphs $G_1$ and $G_2$ by $G_1+G_2=(V(G_1)\cup V(G_2),E(G_1)\cup E(G_2))$. We denote by $sG$ the disjoint union of $s$ copies of $G$, for $s\geq 1$.

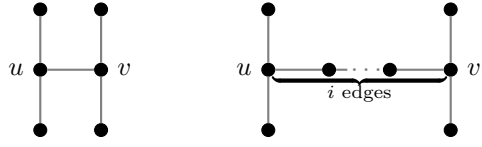
\begin{figure}[t]
\centering
\begin{tikzpicture}
\begin{scope}[scale=0.8]
	\node[vertex] (v1) at (0,2){};
	\node[vertex, label=left:$u$] (v2) at (0,1){};
	\node[vertex] (v3) at (0,0){};
	\node[vertex] (v4) at (1,2){};
	\node[vertex, label=right:$v$] (v5) at (1,1){};
	\node[vertex] (v6) at (1,0){};
	
	\draw[edge](v1)--(v2);
	\draw[edge](v2)--(v5);
	\draw[edge](v2)--(v3);
	\draw[edge](v4)--(v5);
	\draw[edge](v5)--(v6);
	
\end{scope}

\begin{scope}[shift = {(3,0)},scale=0.8]
	\node[vertex] (v1) at (0,2){};
	\node[vertex, label=left:$u$] (v2) at (0,1){};
	\node[vertex] (v3) at (0,0){};
	\node[vertex] (v4) at (3,2){};
	\node[vertex, label=right:$v$] (v5) at (3,1){};
	\node[vertex] (v6) at (3,0){};
	
	\node[vertex] (u1) at (1,1){};
	\node[vertex] (u2) at (2,1){};
	
	\draw[edge](v1)--(v2);
	\draw[edge](v2)--(v3);
	\draw[edge](v4)--(v5);
	\draw[edge](v5)--(v6);
	\draw[edge](v2)--(u1);
	\draw[edge](u2)--(v5);
	\draw[edge](u1)--(1.3,1);
	\node[color = gray](dots) at (1.65,1){$\dots$};
	
	\draw[br](2.925,0.875)--(0.075,0.875);
	\scriptsize
	\node[](i) at (1.5, 0.6){$i$ edges};
	
\end{scope}
\end{tikzpicture}
\caption{The graphs $H_1^*$ (left) and $H_i^*$ (right).}\label{fig-Hstar} 
\end{figure}

Polynomial-time algorithms for {\sc Matching Cut} exist for {\it subcubic} graphs (graphs of maximum degree at most~$3$)~\cite{Ch84}, $K_{1,3}$-free graphs~\cite{Bo09}, $P_6$-free graphs~\cite{LPR22}, $(K_{1,4},K_{1,4}+e)$-free graphs~\cite{KL16} and {\it quadrangulated} graphs, i.e., $(C_5,C_6,\ldots)$-free graphs~\cite{Mo89}; the latter class contains the class of {\it chordal} graphs, i.e., $(C_4,C_5,C_6,\ldots)$-free graphs. Moreover, if {\sc Matching Cut} is polynomial-time solvable for $H$-free graphs, then it is so for $(H+P_3)$-free graphs~\cite{LPR22}.
The problem is \NP-complete even for 
graphs of maximum degree at most~$4$~\cite{Ch84};
$K_{1,4}$-free graphs~\cite{Ch84} (see~\cite{Bo09,KL16}); planar graphs of girth~$5$~\cite{Bo09}; $K_{1,5}$-free bipartite graphs~\cite{LR03}; graphs of girth at least~$g$, for every $g\geq 3$~\cite{FLPR23}; 
 $(3P_5,P_{15})$-free graphs~\cite{LPR23a} (improving a result of~\cite{Fe23}); bipartite graphs where the vertices in one bipartition class all have degree exactly~$2$~\cite{Mo89} and thus for $H_i^*$-free graphs for every odd $i\geq 1$; and for $H_i^*$-free graphs for every even $i\geq 2$~\cite{FLPR23}. 
Recently,  Le and Le~\cite{LL23} proved that {\sc Matching Cut} is \NP-complete even for $(3P_6,2P_7,P_{14})$-free graphs. 
In fact, their hardness gadget also works for {\sc Perfect Matching Cut} (defined below) and {\sc Disconnected Perfect Matching} and can be readily checked to have diameter~$4$ and radius~$3$.
 
The above results imply the following partial complexity classification, which leaves open only a number of cases where $H$ 
is a forest, each connected component of which is either a path or a {\it subdivided claw} (tree with one vertex of degree~$3$ and all other vertices of degree at most~$2$).
 For two graphs $H$ and $H'$, we write $H\ssi H'$ if $H$ is an induced subgraph of $H'$.

\begin{theorem}[\cite{Bo09,Ch84,FLPR23,LL23,LPR23a,LPR22,Mo89}]\label{t-main1}
For a graph~$H$, {\sc Matching Cut} on $H$-free graphs is 
\begin{itemize}
\item polynomial-time solvable if $H\ssi sP_3+K_{1,3}$ or $sP_3+P_6$ for some $s\geq 0$, and
\item \NP-complete if $H\si K_{1,4}$, $P_{14}$, $2P_7$, $3P_5$, $C_r$ for some $r\geq 3$ or $H_j^*$ for some $j\geq 1$.
\end{itemize}
\end{theorem}

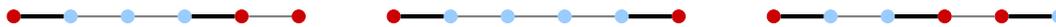
\begin{figure}[t]
\centering
\begin{tikzpicture}
\tikzstyle{cutedge}=[black, ultra thick]
\tikzstyle{matchedge}=[black, ultra thick, dotted]

\begin{scope}[shift={(10,0)}, scale=0.75]
\node[rvertex](p1) at (0,0){};
\node[bvertex](p2) at (1,0){};
\node[bvertex](p3) at (2,0){};
\node[rvertex](p4) at (3,0){};
\node[rvertex](p5) at (4,0){};
\node[bvertex](p6) at (5,0){};

\draw[cutedge](p1)--(p2);
\draw[hedge](p2)--(p3);
\draw[cutedge](p3)--(p4);
\draw[hedge](p4)--(p5);
\draw[cutedge](p5)--(p6);
\end{scope}

\begin{scope}[shift={(5,0)}, scale=0.75]
\node[rvertex](p1) at (0,0){};
\node[bvertex](p2) at (1,0){};
\node[bvertex](p3) at (2,0){};
\node[bvertex](p4) at (3,0){};
\node[bvertex](p5) at (4,0){};
\node[rvertex](p6) at (5,0){};

\draw[cutedge](p1)--(p2);
\draw[hedge](p2)--(p3);
\draw[hedge](p3)--(p4);
\draw[hedge](p4)--(p5);
\draw[cutedge](p5)--(p6);
\end{scope}

\begin{scope}[shift={(0,0)}, scale=0.75]
\node[rvertex](p1) at (0,0){};
\node[bvertex](p2) at (1,0){};
\node[bvertex](p3) at (2,0){};
\node[bvertex](p4) at (3,0){};
\node[rvertex](p5) at (4,0){};
\node[rvertex](p6) at (5,0){};

\draw[cutedge](p1)--(p2);
\draw[hedge](p2)--(p3);
\draw[hedge](p3)--(p4);
\draw[cutedge](p4)--(p5);
\draw[hedge](p5)--(p6);
\end{scope}
\end{tikzpicture}
\caption{The graph $P_6$ with a matching cut of size~$2$ that is not contained in a (disconnected) perfect matching (left), a disconnected perfect matching with a matching cut of size~$2$ (middle) and a perfect matching cut (of size $3$) (right). In each figure, thick edges denote matching cut edges.}\label{f-examples}
\end{figure}

\subsection{Our Focus} 

We already mentioned the known generalization of {\sc Matching Cut} (i.e. {\sc $1$-Cut}) to {\sc $d$-Cut}. In our paper, we consider a different kind of generalization, namely {\sc Maximum Matching Cut}, which is to determine a {\it largest} matching cut of a connected graph (if a matching cut exists). 
So far,
it is known that {\sc Matching Cut} is fixed-parameter tractable when parameterized by the size of the cut; this even holds for {\sc $d$-Cut} for every $d\geq 1$~\cite{AS21,GS21}. However, for special graph classes, {\sc Maximum Matching Cut}
has only been studied for the extreme case, where the task is to decide if a connected graph has a {\it perfect} matching cut which is a matching cut that is even a {\it perfect} matching, i.e., that
 saturates every vertex; see also Figure~\ref{f-examples}. This variant was introduced as {\sc Perfect Matching Cut} by Heggernes and Telle~\cite{HT98}, who proved that it is \NP-complete. We briefly discuss some very recent results for {\sc Perfect Matching Cut} on special graph classes below.

It is readily seen that the gadget in the \NP-hardness reduction of Heggernes and Telle~\cite{HT98} has diameter~$6$ and radius~$3$. We recall that the \NP-hardness gadget of Le and Le~\cite{LL23} for $(3P_6,2P_7,P_{14})$-free graphs even has diameter~$4$ (and radius~$3$). It is also known that {\sc Perfect Matching Cut} is polynomial-time solvable for graphs of radius (and thus also diameter) at most~$2$~\cite{LPR23a}. Hence, we only obtain a dichotomy for graphs of bounded radius but in this case, only a partial complexity classification for graphs of bounded diameter.

\begin{theorem}[\cite{HT98,LL23,LPR23a}]\label{t-diameter2}
For integers $d$ and $r$, {\sc Perfect Matching Cut} for graphs of diameter~$d$ and for graphs of radius~$r$ is polynomial-time solvable if $d\leq 2$ or $r \leq 2$, respectively, and \NP-complete if $d\geq 4$ or $r\geq 3$, respectively.
\end{theorem}

\noindent
For $1\leq h\leq i\leq j$, the graph $S_{h,i,j}$ is the tree of maximum degree~$3$ with exactly one vertex~$u$ of degree~$3$,
whose leaves are at distance~$h$,~$i$ and~$j$, respectively, from~$u$; note $S_{1,1,1}=K_{1,3}$.

It is known that {\sc Perfect Matching Cut} is polynomial-time solvable for 
$S_{1,2,2}$-free graphs (and thus for $K_{1,3}$-free graphs)~\cite{LT22}; $P_6$-free graphs~\cite{LPR23a}; and for pseudo-chordal graphs~\cite{LT22} (and thus for 
chordal graphs, i.e., $(C_4,C_5,\ldots)$-free graphs). Moreover, {\sc Perfect Matching Cut} is polynomial-time solvable for $(H+P_4)$-free graphs if it is polynomial-time solvable for $H$-free graphs~\cite{LPR23a}. It is also known that {\sc Perfect Matching Cut} is \NP-complete even for $3$-connected cubic planar bipartite graphs~\cite{BCD23}, $(3P_6,2P_7,P_{14})$-free graphs~\cite{LL23}, 
 $K_{1,4}$-free bipartite graphs of girth $g$ for every $g\geq 3$~\cite{LT22} and for $H_i^*$-free graphs for every $i\geq 1$~\cite{FLPR23}. This gives us a partial complexity classification:

\begin{theorem}[\cite{FLPR23,LL23,LT22,LPR23a}]\label{t-main2}
For a graph~$H$, {\sc Perfect Matching Cut} on $H$-free graphs is 
\begin{itemize}
\item polynomial-time solvable if $H\ssi sP_4+S_{1,2,2}$ or $sP_4+P_6$ for some $s\geq 0$, and
\item \NP-complete if $H\si K_{1,4}$, $P_{14}$,  $2P_7$, $3P_6$, $C_r$ for some $r\geq 3$ or $H_j^*$ for some $j\geq 1$.
\end{itemize}
\end{theorem}

\noindent
From Theorem~\ref{t-main2} it can be seen that again only cases where $H$ is 
a forest,  each connected component of which is either a path or a subdivided claw,
remain open. However, the number of open cases is smaller than for {\sc Matching Cut},
as can be seen from Theorem~\ref{t-main1}. 
So far, all known complexities for {\sc Matching Cut} and {\sc Perfect Matching Cut} on special graph classes coincide except for (sub)cubic graphs.

We note that whenever {\sc Maximum Matching Cut} is polynomial-time solvable for some graph class, then so are {\sc Matching Cut} and {\sc Perfect Matching Cut}.  Similarly, if one of the latter two problems is \NP-complete, then {\sc Maximum Matching Cut} is \NP-hard. For instance, this immediately yields a complexity dichotomy for graphs of maximum degree at most~$\Delta$. Namely, as {\sc Maximum Matching Cut} is trivial if $\Delta=2$ and {\sc Perfect Matching Cut} is \NP-complete if $\Delta=3$, we have a complexity jump 
from $\Delta=2$ to $\Delta=3$, just like {\sc Perfect Matching Cut}; recall that for {\sc Matching Cut} this jump appears from $\Delta=3$ to $\Delta=4$.
We consider the following research question:

\medskip
\noindent
{\it For which graph classes is {\sc Maximum Matching Cut} harder than {\sc Matching Cut} and {\sc Perfect Matching Cut} and for which graph classes do the complexities coincide?}

\subsection{Our Results for Maximum Matching Cut}\label{s-results}

In Section~\ref{s-hard} we show that {\sc Maximum Matching Cut} is \NP-hard for $2P_3$-free quadrangulated graphs of diameter~$3$ and radius~$2$.
We note that the restrictions to radius~$2$ and diameter~$3$ are not redundant: consider, for example, the $P_6$, which is $2P_3$-free but which has radius~$3$ and diameter~$5$.
In the same section, we also show \NP-hardness for subcubic line graphs of triangle-free graphs, or equivalently, subcubic $(K_{1,3},\mbox{diamond})$-free graphs (the diamond is obtained from the $K_4$ after removing an edge).
These \NP-hardness results are in stark contrast to the situation for {\sc Matching Cut} and {\sc Perfect Matching Cut}, as evidenced by Theorems~\ref{t-diameter}--\ref{t-main2} and to the aforementioned result of Moshi~\cite{Mo89} that  {\sc Matching Cut} is polynomial-time solvable for quadrangulated graphs.

Before proving these results, we first show in Section~\ref{s-poly} that {\sc Maximum Matching Cut} is polynomial-time solvable for graphs of diameter~$2$, generalizing the known polynomial-time algorithms for {\sc Matching Cut} and {\sc Perfect Matching Cut} for graphs of diameter at most~$2$. Hence, all three problems have the same dichotomies for graphs of bounded diameter.

We also prove in Section~\ref{s-poly} that {\sc Maximum Matching Cut} is polynomial-time solvable for $P_6$-free graphs, generalizing the previous polynomial-time results for {\sc Matching Cut} and {\sc Perfect Matching Cut} for $P_6$-free graphs. Due to the hardness result for $2P_3$-free graphs, we cannot show polynomial-time solvability for ``$+P_4$'' (as for {\sc Perfect Matching Cut}) or ``$+P_3$'' (as for {\sc Matching Cut}). However, we can prove that if {\sc Maximum Matching Cut} is polynomial-time solvable for $H$-free graphs, then it is so for $(H+P_2)$-free graphs; again, see Section~\ref{s-poly}.
The common proof technique for our polynomial-time results is as follows:\\[-10pt]
\begin{itemize}
\item [1.] Translate the problem into a colouring problem. We pre-colour some vertices either red or blue, and try to extend the pre-colouring to a red-blue colouring of the whole graph via reduction rules. This technique has been used for {\sc Matching Cut} and {\sc Perfect Matching Cut}, but our analysis is different. In particular, the algorithms for {\sc Matching Cut} and {\sc Perfect Matching Cut} on $P_6$-free graphs use an algorithm for graphs of radius at most~$2$ as a subroutine.
We cannot do this for {\sc Maximum Matching Cut}, as we will show \NP-hardness for radius~$2$.
\item [2.] Reduce the set of uncoloured vertices, via a number of branching steps, to an independent set, and then translate the problem into a matching problem.
This is a new proof ingredient. The matching problem is to find a largest matching that saturates every vertex of the independent set of uncoloured vertices. Plesn\'ik~\cite{Pl99} gave a polynomial time algorithm for this\footnote{The polynomial-time algorithm of Plesn\'ik~\cite{Pl99} solves a more general problem. It takes as input a graph~$G$ with an edge weighting~$w$, a vertex subset $S$ and two integers $a$ and $b$. It then finds a maximum weight matching over all matchings that saturate $S$ and whose cardinality is between $a$ and $b$.}, which we will use as subroutine.\\[-10pt]
\end{itemize}
\noindent
The above polynomial-time and \NP-hardness results yield the following three dichotomies for {\sc Maximum Matching Cut} shown in Section~\ref{s-dicho}; in particular we have obtained a complete complexity classification of {\sc Maximum Matching Cut} for $H$-free graphs (whereas such a classification is only partial for the other two problems, as shown in Theorems~\ref{t-main1} and~\ref{t-main2}).

\begin{theorem}\label{t-dichodiam}
For an integer~$d$, \mmc\ on graphs of diameter $d$ is 
\begin{itemize}
\item polynomial-time solvable if $d \leq 2$, and
\item \NP-hard if $d \geq 3$.
\end{itemize}
\end{theorem}

\begin{theorem}\label{t-dichorad}
For an integer~$r$, \mmc\ on graphs of radius $r$~is 
\begin{itemize}
\item polynomial-time solvable if $r \leq 1$, and
\item \NP-hard if $r \geq 2$.
\end{itemize}
\end{theorem}

\begin{theorem}\label{t-dichoH}
For a graph~$H$, \mmc\ on $H$-free graphs is 
\begin{itemize}
\item polynomial-time solvable if $H\ssi sP_2+P_6$ for some $s\geq 0$, and
\item \NP-hard if $H\si K_{1,3}$, $2P_3$ or $H\si C_r$ for some $r\geq 3$.
\end{itemize}
\end{theorem}

\subsection{A Second Application of Our Proof Techniques}\label{s-second}

In Section~\ref{s-dpm} we apply our techniques on the optimization variant of the problem~{\sc Disconnected Perfect Matching}. 
A {\it disconnected perfect matching} is a perfect matching that contains a matching cut. 
Disconnected perfect matchings were initially studied for cubic graphs from a graph-structural point of view~\cite{Di00,FJLS03}.
The {\sc Disconnected Perfect Matching} problem, which asks whether a given graph has a disconnected perfect matching, was introduced more recently, by Bouquet and Picouleau~\cite{BP} (under a different name\footnote{Bouquet and Picouleau~\cite{BP} use the name {\sc Perfect Matching-Cut} instead of {\sc Disconnected Perfect Matching}. To avoid confusion with {\sc Perfect Matching Cut} we follow the terminology of Le and Telle~\cite{LT22} and use the name {\sc Disconnected Perfect Matching} instead of {\sc Perfect Matching-Cut}. We also note that in the literature the slightly similar name {\sc Disconnected Matching} appears~\cite{GMPFSS23,GHHL05}. However, this name stands for the problem of determining the size of a largest matching whose vertex set induces a disconnected graph, so it is used for a different problem that does not involve edge cuts.}). 
The problem is closely related to {\sc Perfect Matching Cut}. Namely, every perfect matching cut is a disconnected perfect matching. However, the reverse might not be true, as illustrated by the $C_6$, which has a disconnected perfect matching but no perfect matching cut. 

Bouquet and Picouleau~\cite{BP} proved that {\sc Disconnected Perfect Matching} can be solved in polynomial time for graphs of diameter~$2$ and is \NP-complete for graphs of diameter~$3$ (and thus for graphs of radius~$3$). As the problem is trivial for graphs of radius~$1$, this leads to the following classification (in which the case where the radius is $2$ remains open).

\begin{theorem}[\cite{BP}]\label{t-diameter3}
For integers $d$ and $r$, {\sc Disconnected Perfect Matching} for graphs of diameter~$d$ and for graphs of radius~$r$ is polynomial-time solvable if $d\leq 2$ or $r \leq 1$, respectively, and \NP-complete if $d\geq 3$ or $r\geq 3$, respectively.
\end{theorem}

\noindent
Bouquet and Picouleau~\cite{BP} also proved that {\sc Disconnected Perfect Matching} is polynomial-time solvable for bipartite graphs of diameter~$3$, $K_{1,3}$-free graphs and $P_5$-free graphs, and that it is \NP-complete for bipartite graphs of diameter~$4$, $K_{1,4}$-free planar graphs, planar graphs of maximum degree~$4$, planar graphs of girth~$5$, and bipartite $5$-regular graphs. As one of their open problems, they asked about the complexity for 
$P_r$-free graphs for $r\geq 6$. 
In~\cite{LPR23a} we showed that the problem is \NP-complete for $(3P_7,P_{19})$-free graphs. This result was latter improved by Le and Le~\cite{LL23} to $(3P_6,2P_7,P_{14})$-free graphs (we recall that they used the same gadget to prove the complexity of three problems simultaneously). Finally, \NP-completeness has recently been shown for graphs of girth at least~$g$ for all fixed $g\geq 3$~\cite{FLPR23}, and thus for $C_s$-free graphs for all $s\geq 3$. 

We now introduce the {\sc Maximum Disconnected Perfect Matching} problem. This problem is to determine a disconnected perfect matching of a connected graph $G$ with a largest matching cut over all disconnected perfect matchings of $G$. This problem might seem artificial at first sight, but turns out to be highly useful for obtaining results for {\sc Disconnected Perfect Matching}; note that polynomial-time results from the optimization version immediately carry over to the original variant.

By making minor modifications to our proofs, we can show exactly the same results for {\sc Maximum Disconnected Perfect Matching} as for {\sc Maximum Matching Cut}. So, in particular we prove that {\sc Maximum Disconnected Perfect Matching} is polynomial-time solvable for $P_6$-free graphs and for $(H+P_2)$-free graphs, if it is so for $H$-free graphs. Hence, combining these two results with the aforementioned result of~\cite{BP} for $K_{1,3}$-free graphs, we immediately find that {\sc Disconnected Perfect Matching} is polynomial-time solvable for $(K_{1,3}+sP_2)$-free graphs and  $(P_6+sP_2)$-free graphs. This means that we
made further progress on the 
aforementioned open problem of~\cite{BP}.
By combining our new results with the above results from~\cite{BP,FLPR23,LL23}, we can now update the state-of-art summary from~\cite{FLPR23}:

\begin{theorem}\label{t-main3}
For a graph~$H$, {\sc Disconnected Perfect Matching} on $H$-free graphs is 
\begin{itemize}
\item polynomial-time solvable if $H\ssi sP_2+K_{1,3}$ or $sP_2+P_6$ for some $s\geq 0$, and
\item \NP-complete if $H\si K_{1,4}$, $P_{14}$, $3P_6$, $2P_7$, $C_r$ for some $r\geq 3$ or $H_j^*$ for some $j\geq 1$.
\end{itemize}
\end{theorem}

\noindent
Our new results for {\sc Maximum Disconnected Perfect Matching}, proven in Section~\ref{s-dpm}, also lead to the following three dichotomies, as we will show in Section~\ref{s-dpm} as well.

\begin{theorem}\label{t-dichodiam3}
For an integer~$d$,  {\sc Maximum Disconnected Perfect Matching} on graphs of diameter $d$ is 
\begin{itemize}
\item polynomial-time solvable if $d \leq 2$, and
\item \NP-hard if $d \geq 3$.
\end{itemize}
\end{theorem}

\begin{theorem}\label{t-dichorad3}
For an integer~$r$,  {\sc Maximum Disconnected Perfect Matching} on graphs of radius $r$~is 
\begin{itemize}
\item polynomial-time solvable if $r \leq 1$, and
\item \NP-hard if $r \geq 2$.
\end{itemize}
\end{theorem}

\begin{theorem}\label{t-dichoH3}
For a graph~$H$, {\sc Maximum Disconnected Perfect Matching} on $H$-free graphs is 
\begin{itemize}
\item polynomial-time solvable if $H\ssi sP_2+P_6$ for some $s\geq 0$, and
\item \NP-hard if $H\si K_{1,3}$, $2P_3$ or $H\si C_r$ for some $r\geq 3$.
\end{itemize}
\end{theorem}

\noindent
In Section~\ref{s-con} we conclude our paper by stating a number of open problems.

\section{Preliminaries}\label{s-pre}

We consider finite, undirected graphs without multiple edges and self-loops. 
Let $G=(V,E)$ be a connected graph. For $u\in V$, the set $N(u)=\{v \in V\; |\; uv\in E\}$ is the {\it neighbourhood} of $u$ in $G$, where $|N(u)|$ is the {\it degree} of $u$. 
For $S\subseteq V$, the {\it neighbourhood} of $S$ is the set $N(S)=\bigcup_{u\in S}N(u)\setminus S$. 
The graph $G[S]$ is the subgraph of $G$ {\it induced} by $S\subseteq V$, that is, $G[S]$ is the graph obtained from $G$ after deleting the vertices not in $S$. We say that $S$ is a {\it dominating} set of $G$, and that $G[S]$ {\it dominates} $G$ if every vertex of $V\setminus S$ has at least one neighbour in~$S$. The {\it domination number} of $G$ is the size of a smallest dominating set of~$G$.
The set $S$ is an \emph{independent set} if no two vertices in $S$ are adjacent and $S$ is a \emph{clique} if every two vertices in $S$ are adjacent.
A matching $M$ is \emph{$S$-saturating} if every vertex in~$S$ is an end-vertex of an edge in~$M$. 
An $S$-saturating matching is {\it maximum} if there is no $S$-saturating matching of $G$ with more edges. We will use the following result.

\begin{theorem}[\cite{Pl99}]\label{theo-saturatingMatchings}
For a graph $G$ and set $S \subseteq V(G)$, it is possible in polynomial time to find
a maximum $S$-saturating matching or conclude that $G$ has no $S$-saturating matching.
\end{theorem}

\noindent
The \emph{line graph} of a graph $G$ is the graph $L(G)$ whose vertices are the edges of $G$, such that for every two vertices $e$ and $f$, there exists an edge between $e$ and $f$ in $L(G)$ if and only if $e$ and $f$ share an end-vertex in $G$.
A {\it linear forest} is a forest, each connected component of which is a path.
A bipartite graph with non-empty partition classes $V_1$ and $V_2$ is {\it complete} if there is an edge between 
every vertex of $V_1$ and every vertex of $V_2$. If $|V_1|=k$ and $|V_2|=\ell$, then we denote the complete bipartite graph by~$K_{k,\ell}$.
We will need the following theorem.

\begin{theorem}[\cite{HP10}]\label{t-hp}
A graph $G$ on $n$ vertices is $P_6$-free if and only if each connected induced subgraph of $G$ 
contains a dominating induced $C_6$ or 
a dominating (not necessarily induced) complete bipartite graph.
We can find such a dominating subgraph of $G$ in $O(n^3)$ time.
\end{theorem}

\noindent
A {\it red-blue colouring} of a connected graph $G$ colours every vertex of $G$ either red or blue. If every vertex of a set $S\subseteq V$ has the same colour (red or blue), then $S$, and also $G[S]$, are called {\it monochromatic}. An edge with a blue and a red end-vertex is called \emph{bichromatic}.
A red-blue colouring is {\it valid} if every blue vertex has at most one red neighbour; every red vertex has at most one blue neighbour; and both colours red and blue are used at least once. 
A valid red-blue colouring is  {\it perfect-extendable} if there is a perfect matching in $G$ containing all bichromatic edges.
For a valid red-blue colouring of $G$, we let $R$ be the {\it red} set consisting of all vertices coloured red and $B$ be the {\it blue} set consisting of all vertices coloured blue (so $V(G)=R\cup B$). Moreover, the {\it red interface} is the set $R'\subseteq R$ consisting of all vertices in $R$ with a (unique) blue neighbour, and the {\it blue interface} is the set $B'\subseteq B$ consisting of all vertices in $B$ with a (unique) red neighbour in $R$. The {\it value} of a valid red-blue colouring is its number of bichromatic edges, or equivalently, the size of its red (or blue) interface.
A valid red-blue colouring is \emph{maximum} if there is no valid red-blue colouring of the graph with a larger value.
Similarly, a perfect-extendable red-blue colouring is \emph{maximum} if there is no perfect-extendable red-blue colouring of the graph with a larger value.

We can now make the following observations, which can be easily verified (the notion of red-blue colourings has been used before; see, for example,~\cite{Fe23,LPR22}). 

\begin{observation}\label{o} 
For every connected graph $G$ and integer $k$, it holds that 
\begin{itemize}
\item $G$ has a matching cut with at least $k$ edges if and only if $G$ has a valid red-blue colouring of value at least $k$.
\item $G$ has a disconnected perfect matching with at least $k$ edges 
belonging to a matching cut
if and only if $G$ has a perfect-extendable red-blue colouring of value at least $k$.
\end{itemize}
\end{observation}

\begin{observation}\label{lem-cliques-monochrom}
Every complete graph $K_r$ with $r\geq 3$ and every complete bipartite graph~$K_{r,s}$ with $\min\{r,s\}\geq 2$ and $\max\{r,s\}\geq 3$ is monochromatic.
\end{observation}

\section{Polynomial-Time Results for Maximum Matching Cut}\label{s-poly}

In this section we prove three polynomial-time results that we need for obtaining the three dichotomies for {\sc Maximum Matching Cut}, as shown in Theorems~\ref{t-dichodiam}--\ref{t-dichoH}. We first explain our general approach and some helpful lemmas.

The proof of our first lemma for \mmc\ is very similar to the proofs of corresponding lemmas for {\sc Matching Cut}~\cite{Fe23} and {\sc Perfect Matching Cut} \cite{LPR23a}. We include this proof for completeness.
On an aside, the lemma implies that {\sc Maximum Matching Cut} is in \XP\ when parameterized by the domination number of a graph.

\begin{lemma}\label{l-dom}
 For a connected $n$-vertex graph $G$ with domination number~$g$, it is possible to find a maximum red-blue colouring (if a red-blue colouring exists) in $O(2^gn^{g+2})$ time.
\end{lemma}

\begin{proof}
Let $D$ be a dominating set of $G$ with $|D| = g$. We consider all $2^{|D|}= 2^g$ options of colouring the vertices of $D$ red or blue. For every red vertex of $D$ with no blue neighbour, we consider all $O(n)$ options of colouring at most one of its neighbours blue (and thus all of its other neighbours will be coloured red). Similarly, for every blue vertex of $D$ with no red neighbour, we consider all $O(n)$ options of colouring at most one of its neighbours red (and thus all of its other neighbours will be coloured blue). Finally, for every red vertex in~$D$ with already one blue neighbour in $D$, we colour all its yet uncoloured neighbours red. Similarly, for every blue vertex in $D$ with already one red neighbour in $D$, we colour all its yet uncoloured neighbours blue.

As $D$ is a dominating set, the above means that we guessed a red-blue colouring of the whole graph $G$. We can check in $O(n^2)$ time if a red-blue colouring is valid and count its number of bichromatic edges. We take the valid red-blue colouring with largest value. The total number of red-blue colourings that we must consider is $O(2^gn^g)$.
\end{proof}

\noindent
Our general approach is to guess some ``partial'' red-blue colouring that we then try to extend to a maximum valid red-blue colouring of a graph.
To explain this approach we first modify some terminology from \cite{LPR23a} for matching cuts to work for maximum matching cuts as well.

Let $G=(V,E)$ be a connected graph and $S,T,X,Y\subseteq V$ be four non-empty sets with $S\subseteq X$, $T\subseteq Y$ and $X\cap Y=\emptyset$. A {\it red-blue $(S,T,X,Y)$-colouring} of $G$ is a red-blue colouring of 
$G$, 
with a red set containing $X$; a blue set containing $Y$; a red interface containing $S$ and a blue interface containing $T$. 
To obtain a red-blue $(S,T,X,Y)$-colouring, we start with two disjoint subsets $S''$ and $T''$ of $V$, called a {\it starting pair}, such that

\begin{itemize}
\item [(i)] every vertex of $S''$ is adjacent to at most one vertex of $T''$, and vice versa, and 
\item [(ii)] at least one vertex in $S''$ is adjacent to a vertex in $T''$. 
\end{itemize}

\noindent
Let $S^*$ consist of all vertices of $S''$ with a (unique) neighbour in $T''$, and let $T^*$ consist of all vertices of $T''$ with a (unique) neighbour in $S''$; so, every vertex in $S^*$ has a unique neighbour in $T^*$, and vice versa. We call $(S^*,T^*)$ the {\it core} of $(S'',T'')$. Note that $|S^*|=|T^*|\geq 1$. 

We now colour every vertex in $S''$ red and every vertex in $T''$ blue. 
Propagation rules will try to extend $S''$ to a set $X$, and $T''$ to a set $Y$, by finding new vertices whose colour must always be either red or blue. 
That is, we place new red vertices in the set $X$, which already contains~$S''$, and new blue vertices in the set $Y$, which already contains $T''$. If a red and blue vertex are adjacent, then we add the red one to a set $S\subseteq X$ and the blue one to a set $T\subseteq Y$.
So initially, $S:=S^*$, $T:=T^*$, $X:=S''$ and $Y:=T''$. We let $Z:=V\setminus (X\cup Y)$. 

Our task is to try to extend the partial red-blue colouring on $X\cup Y$ to a {\it maximum} valid red-blue $(S,T,X,Y)$-colouring of $G$, that is, a valid red-blue $(S,T,X,Y)$-colouring that has largest value over all valid red-blue $(S,T,X,Y)$-colourings of $G$. 
In order to do this, we present three propagation rules, which indicate necessary implications of previous choices.

We start with rules R1 and R2, which together correspond to the five rules from~\cite{LL19}. 
Rule R1 detects cases where we cannot extend the partial red-blue colouring defined on $X\cup Y$.
Rule R2 tries to extend the sets $S,T,X,Y$ as much as possible. While the sets $S,T,X,Y$ grow, Rule R2 ensures that we keep constructing a (maximum) valid red-blue $(S,T,X,Y)$-colouring (assuming $G$ has a valid red-blue $(S,T,X,Y)$-colouring).

\begin{enumerate}
\item [\bf R1.] Return {\tt no} (i.e., $G$ has no red-blue $(S,T,X,Y)$-colouring) if a vertex $v\in Z$ is
\begin{itemize}
\item [(i)]  adjacent to a vertex in $S$ and to a vertex in $T$, or
\item [(ii)] adjacent to a vertex in $S$ and to two vertices in $Y\setminus T$, or
\item [(iii)]  adjacent to a vertex in $T$ and to two vertices in $X\setminus S$, or
\item [(iv)] adjacent to two vertices in $X\setminus S$ and to two vertices in $Y\setminus T$.
\end{itemize}
\end{enumerate}

\begin{enumerate}
\item [{\bf R2.}] Let $v\in Z$.
\begin{itemize}
\item [(i)] If $v$ is adjacent to a vertex in $S$ or to two vertices of $X\setminus S$, then move $v$ from $Z$ to $X$. If $v$ is also adjacent to a vertex $w$ in $Y$, then add $v$ to $S$ and $w$ to $T$.
\item [(ii)] If $v$ is adjacent to a vertex in $T$ or to two vertices of $Y\setminus T$, then move $v$ from $Z$ to $Y$. If $v$ is also adjacent to a vertex $w$ in $X$, then add $v$ to $T$ and $w$~to~$S$.
\end{itemize}
\end{enumerate}

\noindent
Suppose that exhaustively applying rules R1 and R2 on a starting pair $(S'', T'')$ does not lead to a no-answer but to a tuple $(S', T',X', Y' )$. Then, we call $(S', T',X', Y' )$ an \emph{intermediate tuple}; see also Figure~\ref{fig-lemma-rules-intermediate}.
A propagation rule is \emph{safe} if for every integer $\nu$ the following holds: $G$ has a valid red-blue $(S,T,X,Y)$-colouring of value $\nu$ before the application of the rule if and only if $G$ has a valid red-blue $(S,T,X,Y)$-colouring of value $\nu$ after the application of the rule. 
Le and Le~\cite{LL19} proved the following lemma, which shows that R1 and R2 can be used safely and which is not difficult to verify. The fact that the value $\nu$ is preserved in Lemma~\ref{l-ll1} (ii) below is implicit in their proof.

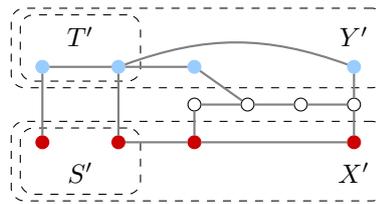
\begin{figure}[ht]
\centering
\begin{tikzpicture}
\tikzstyle{evertex}=[thin,circle,inner sep=0.cm, minimum size=1.7mm, fill=none, draw=black]
\tikzstyle{rahmen1}=[rounded corners = 5pt,draw, dashed, minimum width = 15pt, minimum height = 30pt]
\tikzstyle{rahmen2}=[rounded corners = 5pt,draw, dashed, minimum width = 45pt, minimum height = 25pt]
\tikzstyle{rahmen3}=[rounded corners = 5pt,draw, dashed, minimum width = 140pt, minimum height = 30pt]
	
	\node[rvertex](v1) at (0,0){};
	\node[rvertex](v2) at (1,0){};
	\node[rvertex](v3) at (2,0){};
	\node[evertex](v4) at (2.7,0.5){};
	\node[evertex](v5) at (3.4,0.5){};
	\node[rvertex](v6) at (4.1,0){};
	\node[bvertex](v7) at (0,1){};
	\node[bvertex](v8) at (1,1){};
	\node[bvertex](v9) at (2,1){};
	\node[evertex](v10) at (2,0.5){};
	\node[evertex](v11) at (4.1,0.5){};
	\node[bvertex](v12) at (4.1,1){};
	
	\node[rahmen2](T) at (0.5,1.25){};
	\node[rahmen2](S) at (0.5,-0.25){};
	\node[](Tl) at (0.5, 1.4){$T'$};
	\node[](Sl) at (0.5, -0.4){$S'$};
	\node[rahmen3](S) at (2.05,-0.25){};
	\node[rahmen3](T) at (2.05,1.25){};
	\node[](Tl) at (4.1, 1.4){$Y'$};
	\node[](Sl) at (4.1, -0.4){$X'$};
	
	\draw[hedge](v1)--(v7);
	\draw[hedge](v7)--(v8);
	\draw[hedge](v2)--(v8);
	\draw[hedge](v2)--(v3);
	\draw[hedge](v8)--(v9);
	\draw[hedge](v3)--(v10);
	\draw[hedge](v3)--(v6);
	\draw[hedge](v4)--(v5);
	\draw[hedge](v6)--(v11);
	\draw[hedge](v4)--(v9);
	\draw[hedge](v4)--(v10);
	\draw[hedge](v11)--(v5);
	\draw[hedge](v11)--(v12);
		\path (v8) edge [out=20,in=160, thick, draw=gray] (v12);	
\end{tikzpicture}
\caption{An example (from~\cite{LPR23a}) of a red-blue $(S',T',X',Y')$-colouring of a graph with an intermediate $4$-tuple $(S',T',X',Y')$.}\label{fig-lemma-rules-intermediate}
\end{figure}

\begin{lemma}[\cite{LL19}]\label{l-ll1}
Let $G$ be a connected graph with a starting pair $(S'', T'')$ with core $(S^*, T^*)$, and with an intermediate tuple $(S', T',X', Y' )$. The following three statements hold:\\[-10pt]
\begin{itemize}
\item [(i)] $S^*\subseteq S'$, $S''\subseteq X'$ and $T^*\subseteq T'$, $T''\subseteq Y'$ and $X'\cap Y'=\emptyset$,\\[-10pt]
\item [(ii)] for every integer $\nu$, $G$ has a valid red-blue $(S^*, T^*, S'', T'')$-colouring of value $\nu$ if and only if $G$ has a valid red-blue $(S', T',X', Y' )$-colouring of value $\nu$ (note that the backward implication holds by definition), and\\[-10pt]
\item [(iii)] every vertex in $S'$ has exactly one neighbour in $Y'$, which belongs to $T'$; every vertex in~$T'$ has exactly one neighbour in $X'$, which belongs to $S'$; 
every vertex in $X'\setminus S'$ has no neighbour in $Y'$; every vertex in $Y'\setminus T'$ has no neighbour in $X'$; and every vertex of $V\setminus (X'\cup Y')$ has no neighbour in $S'\cup T'$, at most one neighbour in $X'\setminus S'$, and at most one neighbour in $Y'\setminus T'$.\\[-10pt]
\end{itemize}
Moreover, $(S', T',X', Y' )$ is obtained in polynomial time.
\end{lemma}

\noindent
Let $(S', T',X', Y ')$ be an intermediate tuple of a graph $G$. Let $Z = V \setminus (X' \cup Y' )$. A red-blue
$(S', T',X', Y' )$-colouring of $G$ is called \emph{monochromatic} if all connected components of $G[Z]$ are
monochromatic. 
We say that an intermediate tuple $(S',T',X',Y')$ is \emph{monochromatic} if every connected component of $G[V\setminus (X'\cup Y')]$ is monochromatic in every valid red-blue $(S',T',X',Y')$-colouring of $G$.
A propagation rule is \emph{mono-safe} if for every integer $\nu$ the following holds: $G$ has a valid monochromatic red-blue $(S,T,X,Y)$-colouring of value $\nu$ before the application of the rule if and only if $G$ has a valid monochromatic red-blue $(S,T,X,Y)$-colouring of value~$\nu$ after the application of the rule.

We now present Rule R3 (which is used implicitly in~\cite{LL19}) and prove that R3 is mono-safe.

\begin{itemize}
\item [{\bf R3.}] If there are two distinct vertices $u$ and $v$ in a connected component $D$ of $G[Z]$ with a common neighbour $w\in X \cup Y$, then colour every vertex of $D$ with the colour of $w$.
\end{itemize}

\begin{lemma} \label{lem-monosafe}
Rule R3 is mono-safe.
\end{lemma}

\begin{proof}
Say $w\in X\cup Y$ is in $X$, so $w$ is red. Then, at least one of
$u$ and $v$ must be coloured red. Hence, as $D$ must be monochromatic, every vertex of $D$ must be coloured red. Note that the value of a maximum monochromatic red-blue $(S,T,X,Y)$-colouring (if it exists) is not affected.
\end{proof}

\noindent
Suppose that exhaustively applying rules R1--R3 on an intermediate tuple $(S',T',X',Y')$ does not lead to a no-answer but to a tuple $(S,T,X,Y)$. We call $(S,T,X,Y)$ the \emph{final} tuple. The following lemma can be proved by a straightforward combination of the arguments of the proof of Lemma~\ref{l-ll1} with Lemma~\ref{lem-monosafe} and the observation that an application of R3 takes polynomial time, just as a check to see if R3 can be applied. 

\begin{lemma}\label{l-ll2}
Let $G$ be a connected graph with a monochromatic intermediate tuple $(S', T',X', Y' )$ and a resulting final
tuple $(S, T,X, Y)$. The following three statements hold:\\[-10pt]
\begin{itemize}
\item [(i)] $S'\subseteq S$, $X'\subseteq X$, $T'\subseteq T$, $Y'\subseteq Y$, and $X\cap Y=\emptyset$,\\[-10pt]
\item [(ii)] For every integer $\nu$, $G$ has a valid (monochromatic) red-blue $(S', T',X', Y' )$-colouring of value~$\nu$
if and only if $G$ has a valid monochromatic red-blue $(S, T,X, Y)$-colouring of value $\nu$ (note that the backward implication holds by definition), and\\[-10pt]
\item [(iii)] every vertex in $S$ has exactly one neighbour in $Y$, which belongs to $T$; every vertex in~$T$ has exactly one neighbour in $X$, which belongs to $S$; 
every vertex in $X\setminus S$ has no neighbour in $Y$ and no two neighbours in the same connected component of $G[V\setminus (X\cup Y)]$; every vertex in $Y\setminus T$ has no neighbour in $X$ and no two neighbours in the same connected component of $G[V\setminus (X\cup Y)]$; and
every vertex of $V\setminus (X\cup Y)$ has no neighbour in $S\cup T$, at most one neighbour in $X\setminus S$, and at most one neighbour in $Y\setminus T$.\\[-10pt]
\end{itemize}
Moreover, $(S,T,X,Y)$ is obtained in polynomial time.
\end{lemma}

\noindent
The following lemma will be the final step in each of our polynomial-time results. It is an application of Theorem~\ref{theo-saturatingMatchings}. 

 \begin{figure}[ht]
 \begin{center}
\begin{tikzpicture}

\tikzstyle{rahmen1}=[rounded corners = 5pt,draw, dashed, minimum width = 20pt, minimum height = 75pt]
\tikzstyle{rahmen2}=[rounded corners = 5pt,draw, dashed, minimum width = 15pt, minimum height = 55pt]
\tikzstyle{rahmen3}=[rounded corners = 5pt,draw, dashed, minimum width = 20pt, minimum height = 116pt]
\tikzstyle{rahmen4}=[rounded corners = 5pt,draw, dashed, minimum width = 20pt, minimum height = 60pt]
\begin{scope}[yscale = 0.5, xscale = 0.7 ]
\node[rvertex] (x1) at (0,0){};
\node[rvertex] (x2) at (0,1){};
\node[rvertex] (x3) at (0,2){};
\node[rvertex] (x4) at (0,3){};
\node[rvertex] (x5) at (0,4){};

\node[evertex] (z1) at (2,-2){};
\node[evertex] (z2) at (2,-1){};
\node[evertex] (z3) at (2,0){};
\node[evertex] (z4) at (2,1){};
\node[evertex] (z5) at (2,2){};
\node[evertex] (z6) at (2,3){};
\node[evertex] (z7) at (2,4){};
\node[evertex] (z8) at (2,5){};

\node[bvertex] (y1) at (4,-1){};
\node[bvertex] (y2) at (4,0){};
\node[bvertex] (y3) at (4,1){};
\node[bvertex] (y4) at (4,2){};

\draw[tedge] (x1) -- (z2);
\draw[edge] (x2) -- (z3);
\draw[tedge] (x3) -- (z4);
\draw[edge] (x3) -- (z5);
\draw[tedge] (x5) -- (z6);
\draw[edge] (x5) -- (z7);
\draw[edge] (x5) -- (z8);

\draw[tedge] (y1) -- (z1);
\draw[edge] (y1) -- (z2);
\draw[tedge] (y3) -- (z3);
\draw[edge] (y3) -- (z4);
\draw[tedge] (y4) -- (z5);

\node[rahmen1, nicered] (r) at (0,2){};
\node[rahmen2 ] (r) at (2,0.5){};
\node[rahmen3 ] (r) at (2,1.5){};
\node[rahmen4, lightblue] (r) at (4,0.5){};
\node[nicered](x) at (0,5){$X$};
\node[lightblue](y) at (4,3){$Y$};
\node[](x) at (3,2.5){$U$};
\node[](x) at (2,6){$Z$};

\end{scope}

\begin{scope}[yscale = 0.5, xscale = 0.7, shift = {(7,0)} ]
\node[rvertex] (x1) at (0,0){};
\node[rvertex] (x2) at (0,1){};
\node[rvertex] (x3) at (0,2){};
\node[rvertex] (x4) at (0,3){};
\node[rvertex] (x5) at (0,4){};

\node[rvertex] (z1) at (2,-2){};
\node[bvertex] (z2) at (2,-1){};
\node[rvertex] (z3) at (2,0){};
\node[bvertex] (z4) at (2,1){};
\node[rvertex] (z5) at (2,2){};
\node[bvertex] (z6) at (2,3){};
\node[rvertex] (z7) at (2,4){};
\node[rvertex] (z8) at (2,5){};

\node[bvertex] (y1) at (4,-1){};
\node[bvertex] (y2) at (4,0){};
\node[bvertex] (y3) at (4,1){};
\node[bvertex] (y4) at (4,2){};

\draw[edge] (x1) -- (z2);
\draw[edge] (x2) -- (z3);
\draw[edge] (x3) -- (z4);
\draw[edge] (x3) -- (z5);
\draw[edge] (x5) -- (z6);
\draw[edge] (x5) -- (z7);
\draw[edge] (x5) -- (z8);

\draw[edge] (y1) -- (z1);
\draw[edge] (y1) -- (z2);
\draw[edge] (y3) -- (z3);
\draw[edge] (y3) -- (z4);
\draw[edge] (y4) -- (z5);

\node[rahmen1, nicered] (r) at (0,2){};
\node[rahmen2 ] (r) at (2,0.5){};
\node[rahmen3 ] (r) at (2,1.5){};
\node[rahmen4, lightblue] (r) at (4,0.5){};

\node[nicered](x) at (0,5){$X$};
\node[lightblue](y) at (4,3){$Y$};
\node[](x) at (3,2.5){$U$};
\node[](x) at (2,6){$Z$};

\end{scope}
\end{tikzpicture}
 \caption{A $U$-saturating matching (left) and the corresponding valid red-blue colouring (right). Note that not every vertex in $X \cup Y$ belongs to $W$.}\label{fig-saturatingmatching}
 \end{center}
 \end{figure}
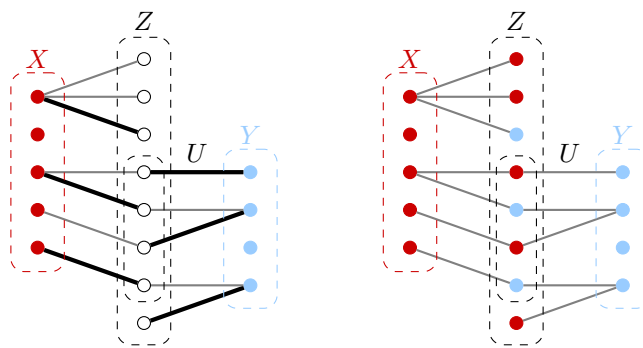
 
 \begin{lemma}\label{lem-plesnikMC}
Let $G=(V,E)$ be a connected graph with a monochromatic intermediate tuple $(S',T',X',Y')$ and a final tuple $(S,T,X,Y)$. If $V\setminus (X\cup Y)$ is an independent set, then it is possible to find in polynomial time either a maximum valid red-blue $(S,T,X,Y)$-colouring of~$G$, or conclude that $G$ has no valid red-blue $(S,T,X,Y)$-colouring.
\end{lemma}
 
 \begin{proof} 
Let $Z=V\setminus (X\cup Y)$. Let $W=N(Z)$. Recall that $Z$ is independent. Hence, by Lemma~\ref{l-ll2}-(iii), every vertex of $W$ belongs to $(X\setminus S) \cup (Y\setminus T)$.
 Let $U\subseteq Z$ consist of all vertices of $Z$ that have a neighbour in both $X\setminus S$ and $Y\setminus T$.
 We claim that the set of bichromatic edges of every valid red-blue $(S,T,X,Y)$-colouring is the union of a $U$-saturating matching in $G[W\cup Z]$ (if it exists) and the set of edges with one end-vertex in $S$ and the other one in $T$. 
 
First suppose that $G[W\cup Z]$ has a $U$-saturating matching~$M$. We colour every vertex in~$X$ red and every vertex in $Y$ blue.
Let $z\in Z$. First assume that $z$ is incident to an edge $zw\in M$. If $w\in X\setminus S$, then colour $z$ blue. If $w\in Y\setminus T$, then colour $z$ red. Now suppose $z$ is not incident to an edge in $M$. Then $z\notin U$, as $M$ is $U$-saturating. Hence, either every neighbour of $z$ belongs to $X\setminus S$ and is coloured red, in which case we colour $z$ red, or every neighbour of $z$ belongs to $Y\setminus T$ and is coloured blue, in which case we colour $z$ blue. 
This gives us a valid red-blue $(S,T,X,Y)$-colouring of $G$. See also Figure~\ref{fig-saturatingmatching}.
 
Now suppose that $G$ has a valid red-blue $(S,T,X,Y)$-colouring. By definition, every vertex of $X$ is coloured red, and every vertex of $Y$ is coloured blue. By Lemma~\ref{l-ll2}-(iii),
every edge with an end-vertex in $S$ and the other one in $T$ is bichromatic, and there are no other bichromatic edges in $G[X\cup Y]$.
Let $M$ be the set of other bichromatic edges. Then, every vertex of $M$ has one vertex in $Z$ and the other one in $W$. Moreover, if $z\in U$, then $z$ has a red neighbour (its neighbour in $X\setminus S$) and a blue neighbour (its neighbour in $Y\setminus T$). Hence, no matter what colour $z$ has itself, $z$ is incident to a bichromatic edge of $M$. We conclude that $M$ is $U$-saturating, and the claim is proven.
 
 From the above claim, it follows that all we have to do is to find a maximum $U$-saturating matching in $G[W\cup Z]$. By Theorem~\ref{theo-saturatingMatchings}, this takes polynomial time.
 \end{proof}
 
 \noindent
 We are now ready to prove our first polynomial-time result.
 
\begin{theorem} \label{theo-P6}
\mmc\ is solvable in polynomial time for $P_6$-free graphs.
\end{theorem}

\begin{proof}
Let $G=(V,E)$ be a connected $P_6$-free graph. By Observation~\ref{o} it suffices to find a maximum valid red-blue colouring of $G$.
By Theorem~\ref{t-hp}, we find in polynomial time either a dominating induced $C_6$ or 
a dominating (not necessarily induced) complete bipartite graph $K_{r,s}$ in $G$. 

If $G$ has a dominating induced $C_6$, then $G$ has domination number at most~$6$, and we apply Lemma~\ref{l-dom}.
Suppose that $G$ has a dominating complete bipartite graph $F$ with partition classes $\{u_1,\ldots,u_r\}$ and
$\{v_1,\ldots,v_s\}$. We may assume without loss of generality that $r\leq s$. If $s\leq 2$, then $G$ has domination number at most~$4$, and we apply Lemma~\ref{l-dom} again. So we assume that $s\geq 3$.

If $r\geq 2$, then $V(F)$ must be monochromatic in any valid red-blue colouring of $G$ by Observation~\ref{lem-cliques-monochrom}. In this case we colour every vertex of $V(F)$ blue.
If $r=1$, then we may assume without loss of generality that $N(u_1)=\{v_1,\ldots,v_s\}$. In this case we colour $u_1$ blue, and we branch over all $O(n)$ options of colouring at most one vertex of $N(u_1)$ red.

So, now we consider a red-blue colouring of $F$. It might be that $F$ is monochromatic (in particular, this will be the case if $r\geq 2$).
If $F$ is monochromatic, then every vertex of $F$ is blue. In order to get a starting pair with a non-empty core, we branch over all $O(n^2)$ options of choosing a bichromatic edge (one end-vertex of which may belong to $F$).
Let $D$ be the set of all coloured vertices, that is, $D$ contains $V(F)$ and possibly one or two other vertices.
By construction, exactly one vertex of $D$ is coloured red, and all other vertices of $D$ are blue.

Let $S^* = S''$ be the set containing the red vertex of $D$. Let $T^*$ be the singleton set containing the blue neighbour of the vertex in $S^*$. Let $T''$ be the set of blue vertices, so $T^* \subseteq T''$. 
We exhaustively apply rules R1 and R2 on the starting pair $(S'',T'')$. By Lemma~\ref{l-ll1} we either find in polynomial time that $G$ has no valid red-blue $(S^*,T^*,S'',T'')$-colouring, and we discard the branch, or 
we obtain an intermediate tuple $(S',T',X',Y')$ of $G$. Suppose the latter case holds. We prove the following two claims for the set $Z' = V\setminus (X' \cup Y')$ of uncoloured vertices. 

\begin{claim1} \label{c-blue}
Every vertex $z\in Z'$ has a neighbour in $Y'\setminus T'$ that belongs to $F$.
\end{claim1}

\begin{claimproof}
As $F$ is dominating, $z$ has a neighbour in $F$.
Since $D\supseteq V(F)$ contains exactly one red vertex $x$, which has a blue neighbour in $D$, all neighbours of $x$ in $G-D$ are coloured red, that is, belong to $X$.
As $z\in G-D$ belongs to $Z'$, this means that $x$ and $z$ are non-adjacent. So, the neighbour of $z$ in $F$ must belong to $Y'\setminus T'$ (as else we could have applied R2).
\end{claimproof}

\begin{claim1} \label{claim-monochrom}
The intermediate tuple $(S',T',X',Y')$ is monochromatic. 
\end{claim1}

\begin{claimproof}
Suppose for a contradiction that there is an edge $uv\in E(G[Z'])$ such that $u$ is blue and $v$ is red. Then $v$ has two blue neighbours by Claim~\ref{c-blue}, a contradiction.
\end{claimproof}
 
 \noindent
Since Claim~\ref{claim-monochrom} holds, we may now exhaustively apply R1--R3 to the intermediate tuple $(S',T',X',Y')$. By Lemma~\ref{l-ll2} we either find in polynomial time that $G$ has no valid red-blue $(S',T',X',Y')$-colouring, and thus no valid red-blue $(S^*,T^*,S',T')$-colouring, and we discard the branch, or we obtain a final tuple $(S,T,X,Y)$ of $G$. Again, we let $Z=V\setminus (X\cup Y)$. 
By the same lemma and Claim~\ref{c-blue}, 
the following holds for every (uncoloured) vertex $z\in Z$:
\begin{itemize}
\item $z$ has at most one neighbour in $X\setminus S$,
\item $z$ has exactly one neighbour in $Y\setminus T$, which belongs to $F$, and
\item if $z'$ is in the same connected component of $G[Z]$ as $z$, then $z$ and $z'$ do not share a neighbour in $G-Z$.
\end{itemize}

\begin{figure}[ht]
\begin{center}
\begin{tikzpicture}
\begin{scope}[scale = 0.5, shift = {(12,0)}]

\node[rvertex](r1) at (0,0){};
\node[rvertex](r2) at (2,0){};

\node[rvertex](r3) at (5,0){};
\node[rvertex](r4) at (7,0){};

\node[bvertex](b1) at (1,4){};
\node[bvertex](b2) at (2.5,3.5){};

\node[bvertex](b3) at (4.5,3.5){};
\node[bvertex](b5) at (3.5, 4.5){};

\node[ellipse, draw=lightblue, minimum width = 100, minimum height = 40](e) at (3.5,4){};
\node[ellipse, draw=nicered, minimum width = 60, minimum height = 40](e) at (1,0){};
\node[ellipse, draw=nicered, minimum width = 60, minimum height = 40](e) at (6,0){};
\node[](x) at (2.8,4.5){$w_3$};
\node[](x) at (1.8,3.5){$w_1$};
\node[](x) at (3.8,3.5){$w_2$};
\node[](x) at (-0.0,-0.52){$z_1'$};
\node[](x) at (2,-0.6){$z_1$};
\node[](x) at (1,-2){$Z_1$};
\node[](x) at (5,-0.6){$z_2$};
\node[](x) at (7,-0.52){$z_2'$};
\node[](x) at (6,-2){$Z_2$};

\draw[edge](r1) -- (b1);
\draw[tedge](r2) -- (b2);
\draw[tedge](r3) -- (b3);
\draw[edge](r4) to [bend right = 30] (b5);

\draw[tedge](r1) -- (r2);
\draw[edge](r3) -- (r4);

\draw[tedge](b2) -- (b5);
\draw[tedge](b3) -- (b5);

\end{scope}

\begin{scope}[scale = 0.5]

\node[rvertex,](r1) at (0,0){};
\node[rvertex, ](r2) at (2,0){};

\node[rvertex](r3) at (5,0){};
\node[rvertex](r4) at (7,0){};

\node[bvertex](b1) at (1,4){};
\node[bvertex](b2) at (2.5,3.5){};

\node[bvertex](b3) at (4.5,3.5){};
\node[bvertex](b4) at (6,4){};
\node[bvertex](b5) at (3.5, 4.5){};

\node[ellipse, draw=lightblue, minimum width = 100, minimum height = 40](e) at (3.5,4){};
\node[ellipse, draw=nicered, minimum width = 60, minimum height = 40](e) at (1,0){};
\node[ellipse, draw=nicered, minimum width = 60, minimum height = 40](e) at (6,0){};
\node[](x) at (2.8,4.5){$w_3$};
\node[](x) at (1.8,3.5){$w_1$};
\node[](x) at (5.2,3.5){$w_2$};
\node[](x) at (-0.0,-0.52){$z_1'$};
\node[](x) at (2,-0.6){$z_1$};
\node[](x) at (1,-2){$Z_1$};
\node[](x) at (5,-0.6){$z_2$};
\node[](x) at (7,-0.52){$z_2'$};
\node[](x) at (6,-2){$Z_2$};

\draw[edge](r1) -- (b1);
\draw[tedge](r2) -- (b2);
\draw[tedge](r3) -- (b3);
\draw[edge](r4) -- (b4);

\draw[tedge](r1) -- (r2);
\draw[tedge](r3) -- (r4);

\draw[tedge](b2) -- (b3);

\end{scope}
\end{tikzpicture}
 \caption{The situation in Claim~\ref{c-one} where two connected components $Z_1, Z_2$ of $G[Z]$, each with at least two vertices, are both coloured red. This will always yield an induced path on at least six vertices, even if $w_1$ and $w_2$ are not adjacent, as at most one of $z_1',z_2'$ is adjacent to $w_3$.}\label{fig-inducedpath}
 \end{center}
 \end{figure}
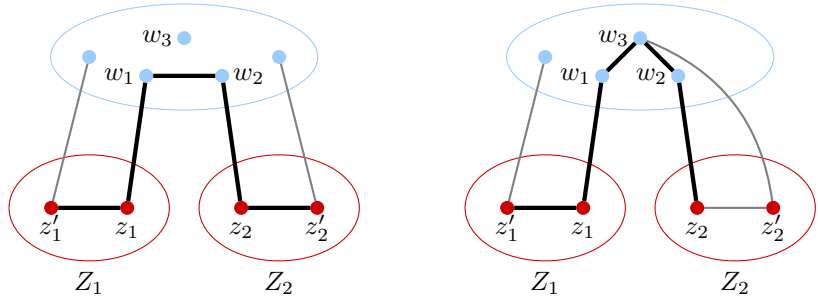
 
\begin{claim1}\label{c-one}
In any valid red-blue $(S,T,X,Y)$-colouring at most one red component of~$G[Z]$ may have more than one vertex.
\end{claim1}

\begin{claimproof}
Let $c$ be a valid red-blue $(S,T,X,Y)$-colouring of $G$.
For a contradiction, assume that $Z_1$ and $Z_2$ are connected components of size at least~$2$ that are both coloured red. For $i=1,2$, let $z_i$ and $z_i'$ be two adjacent vertices in~$Z_i$, and let $w_i$ be the blue neighbour of $z_i$ in~$F$ (which exists due to Claim~\ref{c-blue}).
As $c$ is valid, no blue vertex of $G$ has two red neighbours. Hence, we find that $w_1$ and $w_2$ are distinct vertices, and also that $w_1$ is not adjacent to any vertex of $\{z_1',z_2,z_2'\}$, and $w_2$ is not adjacent to any vertex of $\{z_1,z_1',z_2'\}$.
Hence, if $w_1$ and $w_2$ are adjacent, then $z_1'z_1w_1w_2z_2z_2'$ is an induced $P_6$;
see also Figure~\ref{fig-inducedpath} (left side).
As $G$ is $P_6$-free, this is not possible. Hence, $w_1$ and $w_2$ are not adjacent.

We now use the fact that $w_1$ and $w_2$ both belong to $F$ and that $F$ is a complete bipartite graph. As $w_1w_2\notin E$, the latter means that there exists a vertex $w_3\in V(F)$ that is adjacent to both $w_1$ and $w_2$, so $w_3$ is blue as well. As $z_1'$ and $z_2'$ are both coloured red, at most one of $z_1',z_2'$ can be adjacent to $w_3$. Hence, we may assume without loss of generality that $w_3$ is not adjacent to $z_1'$. As $z_1$ and $z_2$ have $w_1$ and $w_2$, respectively, as their matching partner, $w_3$ is adjacent neither to $z_1$ nor to $z_2$. Now, $z_1'z_1w_1w_3w_2z_2$ is an induced $P_6$, a contradiction.
See also Figure~\ref{fig-inducedpath} (right side).
\end{claimproof}

\noindent
Due to Claim~\ref{c-one}, we can now branch over all $O(n)$ options to colour at most one connected component of $G[Z]$ of size at least~$2$ red, and all other components of size at least~$2$ blue. 
We then exhaustively apply rules R1-R3 again. This takes polynomial time. In essence, we merely pre-coloured some more vertices red. So, in the end we either find a new tuple of $G$ with the same properties as those listed in Lemma~\ref{l-ll2}, or we find that $G$ has no such tuple, in which case we discard the branch. Suppose we have not discarded the branch. Now the set of uncoloured vertices form an independent set. Hence, we can apply Lemma~\ref{lem-plesnikMC} to find in polynomial time a red-blue colouring of $G$ that is a maximum red-blue $(S^*,T^*,S'',T'')$-colouring due to Lemmas~\ref{l-ll1}-(ii) and~\ref{l-ll2}-(ii).

If somewhere in the above process we discarded a branch, that is, if $G$ has no valid red-blue $(S^*,T^*,S'',T'')$-colouring, we consider the next one. If we did not discard the branch, then we remember the value of the maximum red-blue $(S^*,T^*,S'',T'')$-colouring that we found. Afterwards, we pick one with the largest value to obtain a maximum valid red-blue colouring of $G$. 
 
The correctness of our branching algorithm follows from its description. The running time is polynomial: each branch takes polynomial time to process, and the number of branches is $O(n^3)$.
This completes our proof.
\end{proof}

\noindent
The proof of our second polynomial-time result combines Lemma~\ref{lem-plesnikMC} with arguments used in the proof that {\sc Matching Cut} is polynomial-time solvable for $(H+P_3)$-free graphs if it is so for $H$-free graphs~\cite{LPR22}.

\begin{theorem}\label{theo-HP2}
Let $H$ be a graph. If \mmc\ is polynomial-time solvable for $H$-free graphs, then it is so for $(H+P_2)$-free graphs.
\end{theorem}

\begin{proof}
Assume that \mmc\ is polynomial-time solvable for $H$-free graphs. Let $G=(V,E)$ be a connected $(H+P_2)$-free graph on $n$ vertices. If $G$ is $H$-free, we are done by assumption. Suppose $G$ 
has an induced subgraph $G'$ isomorphic to $H$. Let $G^*$ be the graph obtained from $G$ after removing the vertices of $V(G') \cup N(V(G'))$. Since $G'$ is isomorphic to $H$ and $G$ is $(H+P_2)$-free, $G^*$ is $P_2$-free. Hence, $V(G^*)$ is an independent set. By Observation~\ref{o} it suffices to find a maximum valid red-blue colouring of $G$. Below we explain how to do this.

We first branch over all options of colouring every $u\in V(G')$ red or blue, and colouring at most one neighbour of every $u\in V(G')$ with a different colour than $u$. If in a branch we only used one colour, we branch over all $O(n^2)$ options of choosing a bichromatic edge. In this way we obtain, for each branch, a starting pair with a non-empty core.

Consider a branch with a starting pair $(S'',T'')$ and core $(S^*,T^*)$.
We apply rules R1 and R2 exhaustively. If we obtain a no-answer, we may discard the branch due to Lemma~\ref{l-ll1}. Else, we obtain an intermediate tuple $(S,T,X,Y)$. Note that every vertex in $Z=V\setminus (X\cup Y)$ belongs to $G^*$. Hence, $Z$ is an independent set, and thus $(S,T,X,Y)$ is a final tuple. This means that we may apply Lemma~\ref{lem-plesnikMC}. Then, in polynomial time, we either find that $G$ has no valid red-blue $(S,T,X,Y)$-colouring, in which case we may discard the branch due to Lemma~\ref{l-ll1}, or we find a maximum valid red-blue $(S,T,X,Y)$-colouring. The latter is also a maximum valid red-blue $(S^*,T^*,S'',T'')$-colouring, again due to Lemma~\ref{l-ll1}. We remember its value. In the end, after the last branch, we output a colouring with largest value as a maximum valid red-blue colouring of $G$.

The correctness of our branching algorithm follows from its description. The running time is polynomial: each branch takes polynomial time to process, and the number of branches is 
$O(2^{|V(H)|}n^{|V(H)|})+O(n^2)$. This completes our proof.
\end{proof}

\noindent
We now show our third polynomial-time result. Again, the idea is to branch over a polynomial number of options, each of which reduces to the setting where we can apply Lemma~\ref{lem-plesnikMC}.

\begin{theorem}\label{theo-diam}
\mmc \ is solvable in polynomial time for graphs with diameter at most~$2$.
\end{theorem}

\begin{proof}
Let $G=(V,E)$ be a graph of diameter at most~$2$. If $G$ has diameter~$1$, then the problem is trivial to solve. Assume that $G$ has diameter~$2$.
 By Observation~\ref{o} it suffices to find a maximum valid red-blue colouring of $G$.
By definition, such a colouring has at least one bichromatic edge (has value at least~$1$). We branch over all $O(n^2)$ options of choosing the bichromatic edge. 

Consider a branch, where $e=uv$ is the bichromatic edge, say $u$ is blue and $v$ is red. Now all other neighbours of $u$ must be coloured blue.
We let $D=\{u\}\cup N(u)$ and note that $D$ dominates $G$, as $G$ has diameter~$2$. 

We set $S^*=\{v\}$, $T^*=\{u\}$, $S''=\{v\}$, and $T''=D\setminus \{v\}$.
This gives us a starting pair $(S'',T'')$ with core $(S^*,T^*)$.
We exhaustively apply rules R1 and R2 on $(S'',T'')$. By Lemma~\ref{l-ll1} we either find in polynomial time that $G$ has no valid red-blue $(S^*,T^*,S'',T'')$-colouring, and we discard the branch, or 
we obtain an intermediate tuple $(S',T',X',Y')$ of $G$. Suppose the latter case holds. We prove the following two claims for the set $Z' = V\setminus (X' \cup Y')$ of uncoloured vertices. 

\begin{claim1}\label{c-blue2}
Every vertex $z \in Z'$ has a neighbour in $Y'\setminus T'$ that belongs to $D$.
\end{claim1}

\begin{claimproof}
As $z\in Z'$, we have that $z\notin D$.
As $D$ is dominating, $z$ has a neighbour~$b$ in $D$. As every neighbour of $u$ belongs to $D$ and $z$ is not in $D$, we find that $b\neq u$.
Since $D$ contains exactly one red vertex~$v$, which has a blue neighbour in $D$ (namely $u$), all neighbours of $v$ in $G-D$ are coloured red, that is, belong to $X$.
As $z$ belongs to $G-D$ and $z$ is not coloured red, this means that  $v$ and $z$ are non-adjacent, and thus $b\neq v$. So, $b$ must belong to $T''\setminus \{u\}$, and thus to $Y'\setminus T'$, as else we could have applied R2.
\end{claimproof}

\begin{claim1}\label{claim-monochrom2}
The intermediate tuple $(S',T',X',Y')$ is monochromatic. 
\end{claim1}

\begin{claimproof}
Suppose for a contradiction that there is an edge $pq\in E(G[Z'])$ such that $p$ is blue and $q$ is red. Then $q$ has two blue neighbours by Claim~\ref{c-blue2}, a contradiction.
\end{claimproof}

 \noindent
Since Claim~\ref{claim-monochrom2} holds, we may exhaustively apply R1--R3 to the intermediate tuple $(S',T',X',Y')$. By Lemma~\ref{l-ll2} we either find in polynomial time that $G$ has no valid red-blue $(S',T',X',Y')$-colouring, and thus no valid red-blue $(S^*,T^*,S',T')$-colouring, and we discard the branch, or we obtain a final tuple $(S,T,X,Y)$ of $G$. Again, we let $Z=V\setminus (X\cup Y)$. By the same lemma and Claim~\ref{c-blue2}, 
the following holds for every (uncoloured) vertex $w\in Z$:
\begin{itemize}
\item [(i)] $w$ has at most one neighbour in $X\setminus S$,
\item [(ii)] $w$ has exactly one neighbour in $Y\setminus T$, which belongs to $D$, and
\item [(iii)] if $w'$ is in the same connected component of $G[Z]$ as $w$, then $w$ and $w'$ do not share a neighbour in $G-Z$.
\end{itemize}

\noindent
We strengthen (i) by proving the following claim.

\begin{claim1}\label{c-stronger}
Every vertex~$w\in Z$ has exactly one neighbour in $X\setminus S$.
\end{claim1}

\begin{claimproof}
By (i), we find that $w$ has at most one neighbour in $X\setminus S$. For a contradiction, suppose that $w$ has no neighbours in $X\setminus S$. We also know that $w$ has no neighbours in $S$, as else we could have applied R1 or R2. Recall that $v$ was the only red vertex of $D$. As $v$ has a blue neighbour, namely $u$, all the other neighbours of $v$ are coloured red due to R2. Hence, $w$ is adjacent neither to $v$ nor to any vertex in $N(v)\setminus \{u\}$. As all neighbours of $u$ that are not equal to $v$ are coloured blue and $w\in Z$ is uncoloured, we find that $w$ is not adjacent to~$u$ either. Hence, the distance between $v$ and $w$ is at least~$3$, contradicting our assumption that $G$ has diameter~$2$. 
\end{claimproof}

\noindent
We continue by proving the following claim.
 
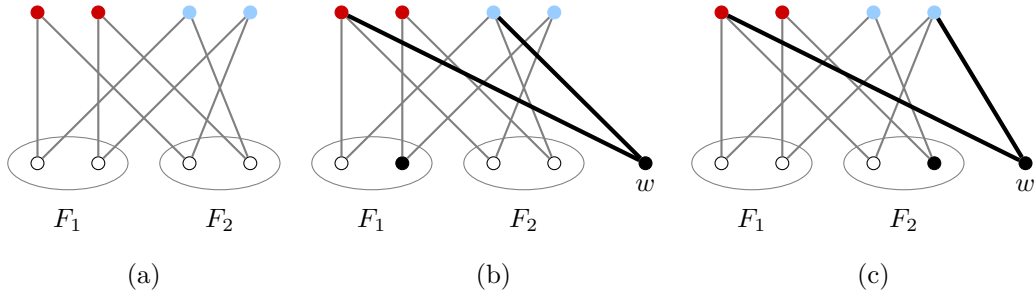
\begin{figure}
\begin{center}
\begin{tikzpicture}

\begin{scope}[xscale = 0.8]

\node[evertex](x1) at (0,0){};
\node[evertex](x2) at (1,0){};
\node[evertex](x3) at (2.5,0){};
\node[evertex](x4) at (3.5,0){};

\node[rvertex](x5) at (0,2){};
\node[rvertex](x6) at (1,2){};
\node[bvertex](x7) at (2.5,2){};
\node[bvertex](x8) at (3.5,2){};

\draw[edge] (x1)--(x5);
\draw[edge] (x1)--(x7);
\draw[edge] (x2)--(x6);
\draw[edge] (x2)--(x8);

\draw[edge] (x3)--(x5);
\draw[edge] (x3)--(x8);
\draw[edge] (x4)--(x6);
\draw[edge] (x4)--(x7);

\node[ellipse, draw=gray, minimum width = 45, minimum height = 20](e) at (0.5,0){};
\node[ellipse, draw=gray, minimum width = 45, minimum height = 20](e) at (3,0){};

\node[](x) at (0.5,-0.75){$F_1$};
\node[](x) at (3,-0.75){$F_2$};
\node[](x) at (1.75,-1.5){(a)};
\end{scope}

\begin{scope}[shift = {(4,0)}, xscale = 0.8]

\node[evertex](x1) at (0,0){};
\node[vertex](x2) at (1,0){};
\node[evertex](x3) at (2.5,0){};
\node[evertex](x4) at (3.5,0){};

\node[rvertex](x5) at (0,2){};
\node[rvertex](x6) at (1,2){};
\node[bvertex](x7) at (2.5,2){};
\node[bvertex](x8) at (3.5,2){};

\node[vertex, label=below:$w$](x9) at (5,0){};

\draw[edge] (x1)--(x5);
\draw[edge] (x1)--(x7);
\draw[edge] (x2)--(x6);
\draw[edge] (x2)--(x8);

\draw[edge] (x3)--(x5);
\draw[edge] (x3)--(x8);
\draw[edge] (x4)--(x6);
\draw[edge] (x4)--(x7);

\draw[tedge] (x9)--(x5);
\draw[tedge] (x9)--(x7);

\node[ellipse, draw=gray, minimum width = 45, minimum height = 20](e) at (0.5,0){};
\node[ellipse, draw=gray, minimum width = 45, minimum height = 20](e) at (3,0){};
\node[](x) at (0.5,-0.75){$F_1$};
\node[](x) at (3,-0.75){$F_2$};
\node[](x) at (2.5,-1.5){(b)};

\end{scope}

\begin{scope}[shift = {(9,0)}, xscale = 0.8]

\node[evertex](x1) at (0,0){};
\node[evertex](x2) at (1,0){};
\node[evertex](x3) at (2.5,0){};
\node[vertex](x4) at (3.5,0){};

\node[rvertex](x5) at (0,2){};
\node[rvertex](x6) at (1,2){};
\node[bvertex](x7) at (2.5,2){};
\node[bvertex](x8) at (3.5,2){};

\node[vertex, label=below:$w$](x9) at (5,0){};

\draw[edge] (x1)--(x5);
\draw[edge] (x1)--(x7);
\draw[edge] (x2)--(x6);
\draw[edge] (x2)--(x8);

\draw[edge] (x3)--(x5);
\draw[edge] (x3)--(x8);
\draw[edge] (x4)--(x6);
\draw[edge] (x4)--(x7);

\draw[tedge] (x9)--(x5);
\draw[tedge] (x9)--(x8);

\node[ellipse, draw=gray, minimum width = 45, minimum height = 20](e) at (0.5,0){};
\node[ellipse, draw=gray, minimum width = 45, minimum height = 20](e) at (3,0){};
\node[](x) at (0.5,-0.75){$F_1$};
\node[](x) at (3,-0.75){$F_2$};
\node[](x) at (2.5,-1.5){(c)};

\end{scope}
\end{tikzpicture}
 \caption{The unique way (up to symmetry) to connect two components $F_1$ and $F_2$ of $G[Z]$ of size 2 (a) and the two options to connect a vertex $w$ in a third component $F_3$ to the coloured part of the graph (b) and (c). We can see that there is always an uncoloured vertex without a common neighbour with $w$. }\label{fig-diameter2}
 \end{center}
 \end{figure}
 
 \begin{claim1}\label{claim-diam3}
If $G[Z]$ contains two connected components $F_1$ and $F_2$ of size at least~$2$, then $G[Z]=F_1+F_2$.
 \end{claim1}
 
 \begin{claimproof}
 Let $F_1$ contain $u_1$ and $u_2$. Let $F_2$ contain $v_1$ and $v_2$. By combining Claim~\ref{c-stronger} with (ii) and (iii), we find that the vertices $u_1$ and $u_2$ have each a different red (respectively, blue) neighbour and the same holds for $v_1$ and $v_2$.
However, as $G$ has diameter~$2$, it holds that $u_1$ and $u_2$ each have a common neighbour with both $v_1$ and $v_2$. Thus, without loss of generality, $u_1$ and $v_1$ have a red common neighbour, $u_1$ and $v_2$ a blue one, while $u_2$ and $v_1$ have a blue common neighbour, $u_2$ and $v_2$ a red one. See also Figure~\ref{fig-diameter2} (a). 
 
For a contradiction, assume that $G[Z]$ contains a third connected component $F_3$. Let $w$ be a vertex in $F_3$. Then $w$ has a common neighbour with each of $u_1,u_2,v_1$ and $v_2$. Furthermore, $w$ has exactly one red and one blue neighbour. As can be seen in Figures~\ref{fig-diameter2}~(b) and~(c), there do not exist vertices $x \in X$ and $y \in Y$ such that $\{u_1,u_2,v_1, v_2\} \subseteq N_G(\set{x,y})$. Hence, $w$ has no common neighbour with some vertex of $\{u_1,u_2,v_1,v_2\}$, contradicting our assumption that $G$ has diameter~$2$.
 \end{claimproof}
 
 \noindent
 From Claim~\ref{claim-diam3}, it follows that $G[Z]$ has at most two components with more than one vertex, which are both monochromatic in every valid red-blue $(S,T,X,Y)$-colouring of $G$ (if such a colouring exists) due to Claim~\ref{claim-monochrom2}. Hence, we can branch over all possible colourings of these connected components (there are at most four branches). 
 
 For each branch, we propagate the obtained partial red-blue colouring by exhaustively applying rules R1--R3. 
 This takes polynomial time. In essence, we merely pre-coloured some more vertices red or blue. So, in the end we either find a new tuple of $G$ with the same properties as those listed in Lemma~\ref{l-ll2}, or we find that $G$ has no such tuple, in which case we discard the branch. Suppose we have not discarded the branch. Now the set of uncoloured vertices form an independent set. Hence, we can apply Lemma~\ref{lem-plesnikMC} to find in polynomial time a red-blue colouring of $G$ that is a maximum red-blue $(S^*,T^*,S'',T'')$-colouring due to Lemmas~\ref{l-ll1}-(ii) and~\ref{l-ll2}-(ii).

If somewhere in the above process we discarded a branch, that is, if $G$ has no valid red-blue $(S^*,T^*,S'',T'')$-colouring, we consider the next one. If we did not discard the branch, then we remember the value of the maximum red-blue $(S^*,T^*,S'',T'')$-colouring that we found. Afterwards, we pick one with the largest value to obtain a maximum valid red-blue colouring of $G$. 
 
The correctness of our branching algorithm follows from its description. The running time is polynomial: each branch takes polynomial time to process, and the number of branches is $O(n^2)$.
This completes our proof.
\end{proof}

\section{Hardness Results for Maximum Matching Cut}\label{s-hard}

In the following we will prove that \mmc\ is \NP-hard for subcubic line graphs and $2P_3$-free quadrangulated graphs of diameter~$3$ and radius~$2$.
To prove the first hardness result, we reduce from {\sc Maximum Cut}. The problem takes as input a graph $G$ and an integer $k$. The question is whether $G$ has an edge cut of size at least $k$. This problem is well known to be \NP-complete even for subcubic graphs, as shown by Yannakakis~\cite{Ya78}.

  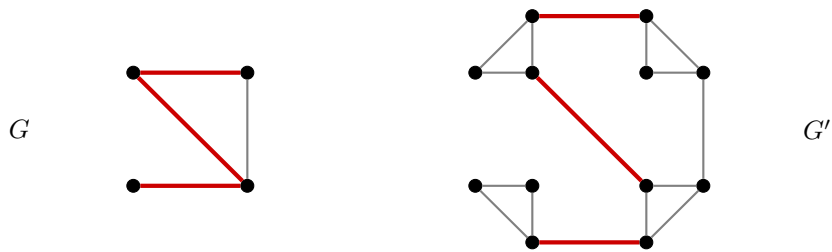
\begin{figure}[h]
  \vspace*{5mm}
 \begin{center}
 \begin{tikzpicture}[scale=1.5]
\node[] (g) at (-1,0.5){$G$};

\node[vertex] (v1) at (0,0){};
\node[vertex] (v2) at (1,0){};
\node[vertex] (v3) at (0,1){};
\node[vertex] (v4) at (1,1){};

\draw[tredge] (v1)--(v2);
\draw[tredge] (v3)--(v2);
\draw[edge] (v4)--(v2);
\draw[tredge] (v3)--(v4);

\begin{scope}[shift = {(3,-0.5)}, scale = 0.5]
\node[] (g) at (6,2){$G'$};

\node[vertex] (v11) at (1,0){};
\node[vertex] (v12) at (1,1){};
\node[vertex] (v13) at (0,1){};

\node[vertex] (v21) at (3,0){};
\node[vertex] (v22) at (3,1){};
\node[vertex] (v23) at (4,1){};

\node[vertex] (v31) at (0,3){};
\node[vertex] (v32) at (1,3){};
\node[vertex] (v33) at (1,4){};

\node[vertex] (v41) at (4,3){};
\node[vertex] (v42) at (3,3){};
\node[vertex] (v43) at (3,4){};

\draw[edge] (v11)--(v12);
\draw[edge] (v12)--(v13);
\draw[edge] (v13)--(v11);

\draw[edge] (v21)--(v22);
\draw[edge] (v22)--(v23);
\draw[edge] (v23)--(v21);

\draw[edge] (v31)--(v32);
\draw[edge] (v32)--(v33);
\draw[edge] (v33)--(v31);

\draw[edge] (v41)--(v42);
\draw[edge] (v42)--(v43);
\draw[edge] (v43)--(v41);

\draw[tredge] (v11)--(v21);
\draw[tredge] (v22)--(v32);
\draw[edge] (v23)--(v41);
\draw[tredge] (v43)--(v33);

\end{scope}
\end{tikzpicture}
 \caption{A graph $G$ (left) where the tick red edges form a maximum edge cut, and the graph~$G'$ (right) from the proof of Theorem~\ref{theo-clawfree}, where the thick red edges form a maximum matching cut.}\label{fig-clawfree}
 \end{center}
 \end{figure}
 
 \begin{theorem}\label{theo-clawfree}
\mmc \ is \NP-hard for subcubic line graphs of triangle-free graphs.
\end{theorem}
 
\begin{proof}
Let $ (G,k)$ be an instance of {\sc Maximum Cut}, where $G$ is a subcubic graph.
From $G$, we construct a graph $G'$ as follows. First replace every vertex $v \in V(G)$ by a triangle $C_v$. 
Next, for every edge $uv \in E(G)$, add an edge between a vertex in $C_v$ and a vertex in $C_u$, such that every vertex in $C_v$ has at most one neighbour outside of $C_v$. This is possible since $G$ is subcubic. See Figure~\ref{fig-clawfree} for an example. The graph $G'$ is subcubic, as every vertex in $G'$ has two neighbours inside a triangle and at most one neighbour outside.
Moreover, $G$ is $(K_{1,3},\mbox{diamond})$-free, or equivalently, the line graph of a triangle-free graph. 

We claim that $G$ has an edge cut of size at least~$k$ if and only if $G'$ has a matching cut of size at least $k$.

First suppose that $G$ has an edge cut~$M$ of size at least~$k$. So, $V(G)$ can be partitioned into sets $R$ and $B$, such that for every $e\in E(G)$, it holds that $e\in C$ if and only if $e$ has one end-vertex in $R$ and the other one in $B$. We define the edge set 
$$M' = \set{ u'v'\in E(G')\; |\; u' \in C_u, v'\in C_v, uv \in M}.$$
Note that $|M'|=|M|\geq k$. Moreover, $M'$ contains no edge from any triangle $C_u$, so $M'$ is a matching. For every $v\in V(G)$, we put all vertices of $C_v$ in a set $B'$ if $v\in B$, and else we put all vertices of $C_v$ in a set $R'$. We now find that for every edge $e\in E(G')$, it holds that $e$ belongs to $M'$ if and only if $e$ has one end-vertex in $R'$ and the other one in $B'$. Hence, $M'$ is an edge cut, and thus a matching cut, of $G'$ with $|M'|\geq k$. 

Now suppose that $G'$ has a matching cut~$M'$ of size at least $k$. Let $R'$ and $B'$ be the corresponding sets of red and blue vertices, respectively.
We define the edge set 
$$M= \set{ uv \in E(G)\;|\; u'v' \in M', u' \in C_u, v' \in C_v}.$$
Note that $|M|=|M'|\geq k$.
Due to Lemma~\ref{lem-cliques-monochrom}, every triangle $C_u$ is monochromatic. For every $u\in V(G)$, we put $u$ in a set $R$ if $C_u$ is coloured red, else we put $u$ in a set $B$.
We now find for every edge $e\in E(G)$ that $e$ belongs to $M$ if and only if one end-vertex of $e$ belongs to $R$ and the other one to $B$.
Hence, $M$ is an edge cut in $G$ of size at least~$k$.
\end{proof}

\noindent
For our next \NP-hardness result, we reduce from the following problem. An {\it exact $3$-cover} of a set $X$ is a collection ${\cal C}$ of $3$-element subsets of $X$, such that every 
$x\in X$ is in exactly one $3$-element subset of ${\cal C}$. The {\sc Exact $3$-Cover} problem has as input a set $X$ with $3q$ elements and a collection $\mathcal{S}$ of $3$-element subsets of $X$.
The question is if $\mathcal{S}$ contains an {\it exact $3$-cover} of $X$ (which will be of size $q$). This problem is well known to be \NP-complete (see \cite{Ka72}).

 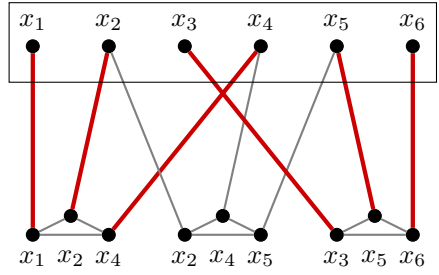
\begin{figure}[h]
 \vspace*{5mm}
 \begin{center}
\begin{tikzpicture}
\begin{scope}[scale = 0.5]

\node[vertex, label=above:$x_1$](x1) at (0,5){};
\node[vertex, label=above:$x_2$](x2) at (2,5){};
\node[vertex, label=above:$x_3$](x3) at (4,5){};
\node[vertex, label=above:$x_4$](x4) at (6,5){};
\node[vertex, label=above:$x_5$](x5) at (8,5){};
\node[vertex, label=above:$x_6$](x6) at (10,5){};

\node[vertex, label=below:$x_1$](c11) at (0,0){};
\node[vertex, label={[label distance=7]below:$x_2$}](c12) at (1,0.5){};
\node[vertex, label=below:$x_4$](c13) at (2,0){};

\node[vertex, label=below:$x_2$](c21) at (4,0){};
\node[vertex,  label={[label distance=7]below:$x_4$}](c22) at (5,0.5){};
\node[vertex, label=below:$x_5$](c23) at (6,0){};

\node[vertex, label=below:$x_3$](c31) at (8,0){};
\node[vertex,  label={[label distance=7]below:$x_5$}](c32) at (9,0.5){};
\node[vertex, label=below:$x_6$](c33) at (10,0){};

\draw[edge](c11)--(c12);
\draw[edge](c12)--(c13);
\draw[edge](c11) -- (c13);

\draw[edge](c21)--(c22);
\draw[edge](c22)--(c23);
\draw[edge](c21) -- (c23);

\draw[edge](c31)--(c32);
\draw[edge](c32)--(c33);
\draw[edge](c31) -- (c33);

\draw[tredge](c11)--(x1);
\draw[tredge](c12)--(x2);
\draw[tredge](c13)--(x4);

\draw[edge](c21)--(x2);
\draw[edge](c22)--(x4);
\draw[edge](c23)--(x5);

\draw[tredge](c31)--(x3);
\draw[tredge](c32)--(x5);
\draw[tredge](c33)--(x6);

\node[rectangle, minimum width = 160,  minimum height = 30, draw = black](r) at (5,5.1){};
\end{scope}
\end{tikzpicture}
 \caption{The graph $G$ for $X = \set{x_1,\dots, x_6}$ and $\mathcal{S} = \set{\set{x_1,x_2,x_4}, \set{x_2, x_4, x_5}, \set{x_3, x_5, x_6}}$. The vertices in the rectangle form a clique, whose edges we did not draw for readability. For the same reason, we also omitted the superscripts of the vertices in the three triangles. The set $\mathcal{S'} =\set{\set{x_1,x_2,x_4}, \set{x_3, x_5, x_6}}$ is an exact $3$-cover of $X$. The thick red edges in the graph show the corresponding matching cut of size~$3q=6$.}\label{fig-2P3}
 \end{center}
 \end{figure}
 
\begin{theorem}\label{theo-2P3}
\mmc\ is \NP-hard for $2P_3$-free 
quadrangulated 
graphs of radius at most~$2$ and diameter at most~$3$.
\end{theorem}

\begin{proof}
Let $(X,\mathcal{S})$ be an instance of {\sc Exact $3$-Cover} where $X=\{x_1,\ldots,x_{3q}\}$ and ${\cal S}=\{S_1,\ldots,S_k\}$, such that each $S_i$ contains exactly three elements of $X$.
 From $(X,{\cal S})$ we construct a graph $G$. We first define a clique $K_X=\{x_1,\ldots,x_{3q}\}$. For each $S\in {\cal S}$, we do as follows. Let $S=\{x_h,x_i,x_j\}$. We add a triangle $K_S$ on vertices $x^S_h$, $x^S_i$ and $x^S_j$. We add an edge between a vertex $x_i\in K_X$ and a vertex $u\notin K_X$ if and only if $u=x_i^S$ for some $S\in {\cal S}$. This completes the construction of $G$. See Figure~\ref{fig-2P3} for an example.

As every induced $P_3$ must contain at least one vertex from the clique $K_X$, we find that $G$ is $2P_3$-free. As $G$ is not only $2P_3$-free, but also $(C_5,C_6)$-free, $G$ is quadrangulated.
Consider some $x_i\in K_X$. Then every other vertex is of distance at most~$2$ from $x_i$. 
Consider some $x^S_i\in K_S$ for some $S\in {\cal S}$. Then every other vertex is of distance at most~$3$ from $x^S_i$.
Hence, the radius of $G$ is at most~$2$ and the diameter of $G$ is at most~$3$. 

We claim that ${\cal S}$ contains an exact $3$-cover of $X$ if and only if $G$ has a matching cut of size~$3q$.
First suppose that ${\cal S}$ contains an exact $3$-cover ${\cal C}$ of $X$. We colour every vertex of $K_X$ red. We colour a triangle $K_S$ blue if $S\in {\cal S}$ and otherwise we colour it red. This yields a valid red-blue colouring of value~$3q$, and thus a matching cut of size~$3q$.

Now suppose that $G$ has a matching cut~$M$ of size~$3q$. As $K_X$ is a clique of size $3q\geq 3$, the corresponding valid red-blue colouring assigns every vertex of $K_X$ the same colour, say red. As every triangle $K_S$ is monochromatic, this means that exactly $q$ triangles must be coloured blue. Moreover, no two blue triangles have a common red neighbour in $K_X$. Hence, the blue triangles correspond to an exact $3$-cover of $X$. See again Figure~\ref{fig-2P3}.
\end{proof}

\section{Dichotomies for Maximum Matching Cut}\label{s-dicho}

In this section we prove our three dichotomy results, which we restate below.

\medskip
\noindent
{\bf Theorem~\ref{t-dichodiam} (restated).}
{\it For an integer~$d$, \mmc\ on graphs of diameter $d$ is 
\begin{itemize}
\item polynomial-time solvable if $d \leq 2$, and
\item \NP-hard if $d \geq 3$.
\end{itemize}}

\begin{proof}
The two results follow from Theorems~\ref{theo-diam} and~\ref{theo-2P3}, respectively.
\end{proof}

\noindent
{\bf Theorem~\ref{t-dichorad} (restated).}
{\it For an integer~$r$, \mmc\ on graphs of radius~$r$~is 
\begin{itemize}
\item polynomial-time solvable if $r \leq 1$, and
\item \NP-hard if $r \geq 2$.
\end{itemize}}

\begin{proof}
A graph of radius~$1$ has a dominating vertex, and thus it has a matching cut if and only if it has a vertex of degree~$1$. This can be checked in polynomial time and thus proves the first result. The second result follows from Theorem~\ref{theo-2P3}.
\end{proof}

\noindent
{\bf Theorem~\ref{t-dichoH} (restated).}
{\it For a graph~$H$, \mmc\ on $H$-free graphs is 
\begin{itemize}
\item polynomial-time solvable if $H\ssi sP_2+P_6$ for some $s\geq 0$, and
\item \NP-hard if $H\si K_{1,3}$, $2P_3$ or $H\si C_r$ for some $r\geq 3$.
\end{itemize}}

\begin{proof}
Let $H$ be a graph. If $H$ contains a cycle, then {\sc Matching Cut}, and thus {\sc Maximum Matching Cut}, is \NP-hard due to Theorem~\ref{t-main1}.
Now suppose that $H$ has no cycle, so $H$ is a forest. If $H$ contains a vertex of degree at least~$3$, then the class of $H$-free graphs contains the class of $K_{1,3}$-free graphs. The latter class contains the class of line graphs, and thus we can apply Theorem~\ref{theo-clawfree}. 

Now suppose that $H$ is a forest of maximum degree at most~$2$, that is, $H$ is a linear forest. If $H\ssi sP_2+P_6$ for some $s\geq 0$, then we apply 
Theorem~\ref{theo-P6} in combination with $s$ applications of Theorem~\ref{theo-HP2}.
Else $H$ contains an induced $2P_3$ and we apply Theorem~\ref{theo-2P3}. This completes the proof.
\end{proof}

\section{Dichotomies for Maximum Disconnected Perfect Matching}\label{s-dpm}

In this section we proof the results of Section~\ref{s-second}.
We first need a similar lemma as Lemma~\ref{l-dom}, which is proven by copying the arguments of the proof of Lemma~\ref{l-dom} and using the fact that we can check if a graph has a perfect matching in polynomial time by using, for instance, 
Edmonds' Blossom algorithm.

\begin{lemma}\label{l-dom2}
 For a connected $n$-vertex graph $G$ with domination number~$g$, it is possible to find a maximum perfect-extendable red-blue colouring (if a red-blue colouring exists) in $O(2^gn^{g+2})$ time.
\end{lemma}

\noindent
Note that Lemmas~\ref{l-ll1} and~\ref{l-ll2} can be adapted to perfect-extendable red-blue colourings in a straightforward way, since the propagation rules R1--R3 hold for valid red-blue colourings and every perfect-extendable red-blue colouring is valid (that is, the propagation of any partial red-blue colouring does not influence if the resulting partial red-blue colouring is perfect-extendable or not). In other words, we immediately obtain the following two lemmas.

\begin{lemma}\label{l-ll1b}
Let $G$ be a connected graph with a starting pair $(S'', T'')$ with core $(S^*, T^*)$, and with an intermediate tuple $(S', T',X', Y' )$. The following three statements hold:\\[-10pt]
\begin{itemize}
\item [(i)] $S^*\subseteq S'$, $S''\subseteq X'$ and $T^*\subseteq T'$, $T''\subseteq Y'$ and $X'\cap Y'=\emptyset$,\\[-10pt]
\item [(ii)] for every integer $\nu$, $G$ has a perfect-extendable red-blue $(S^*, T^*, S'', T'')$-colouring of value~$\nu$ if and only if $G$ has a perfect-extendable red-blue $(S', T',X', Y' )$-colouring of value~$\nu$ (note that the backward implication holds by definition), and\\[-10pt]
\item [(iii)] every vertex in $S'$ has exactly one neighbour in $Y'$, which belongs to $T'$; every vertex in~$T'$ has exactly one neighbour in $X'$, which belongs to $S'$; 
every vertex in $X'\setminus S'$ has no neighbour in $Y'$; every vertex in $Y'\setminus T'$ has no neighbour in $X'$; and every vertex of $V\setminus (X'\cup Y')$ has no neighbour in $S'\cup T'$, at most one neighbour in $X'\setminus S'$, and at most one neighbour in $Y'\setminus T'$.\\[-10pt]
\end{itemize}
Moreover, $(S', T',X', Y' )$ is obtained in polynomial time.
\end{lemma}

\begin{lemma}\label{l-ll2b}
Let $G$ be a connected graph with a monochromatic intermediate tuple $(S', T',X', Y' )$ and a resulting final
tuple $(S, T,X, Y)$. The following three statements hold:\\[-10pt]
\begin{itemize}
\item [(i)] $S'\subseteq S$, $X'\subseteq X$, $T'\subseteq T$, $Y'\subseteq Y$, and $X\cap Y=\emptyset$,\\[-10pt]
\item [(ii)] For every integer $\nu$, $G$ has a perfect-extendable (monochromatic) red-blue $(S', T',X', Y' )$-colouring of value~$\nu$
if and only if $G$ has a perfect-extendable monochromatic red-blue $(S, T,X, Y)$-colouring of value $\nu$ (note that the backward implication holds by definition), and\\[-10pt]
\item [(iii)] every vertex in $S$ has exactly one neighbour in $Y$, which belongs to $T$; every vertex in~$T$ has exactly one neighbour in $X$, which belongs to $S$; 
every vertex in $X\setminus S$ has no neighbour in $Y$ and no two neighbours in the same connected component of $G[V\setminus (X\cup Y)]$; every vertex in $Y\setminus T$ has no neighbour in $X$ and no two neighbours in the same connected component of $G[V\setminus (X\cup Y)]$; and
every vertex of $V\setminus (X\cup Y)$ has no neighbour in $S\cup T$, at most one neighbour in $X\setminus S$, and at most one neighbour in $Y\setminus T$.\\[-10pt]
\end{itemize}
Moreover, $(S,T,X,Y)$ is obtained in polynomial time.
\end{lemma}

\noindent
We prove our next lemma by similar but simpler arguments as in the proof of Lemma~\ref{lem-plesnikMC}.
 
\begin{lemma}\label{lem-plesnikDPM}
Let $G=(V,E)$ be a connected graph with a final tuple $(S,T,X,Y)$. If $V\setminus (X\cup Y)$ is an independent set, then it is possible to find in polynomial time either a maximum perfect-extendable red-blue $(S,T,X,Y)$-colouring of $G$, or conclude that $G$ has no perfect-extendable red-blue $(S,T,X,Y)$-colouring.
\end{lemma}
 
\begin{proof} 
Let $Z=V\setminus (X\cup Y)$. Let $W=N(Z)$. Recall that $Z$ is independent. Hence, by Lemma~\ref{l-ll2b}-(iii), every vertex of $W$ belongs to $(X\setminus S) \cup (Y\setminus T)$.
By combining this observation with Lemma~\ref{l-ll2b}-(iii), we find that every bichromatic edge of any perfect-extendable red-blue $(S,T,X,Y)$-colouring (if it exists) is 

\begin{itemize}
\item [(i)] either an edge with one end-vertex in $S$ and the other one in $T$, or
\item [(ii)] an edge with one end-vertex in $Z$ and the other one in either $X\setminus S$ or $Y\setminus T$.
\end{itemize}

\noindent
 By Lemma~\ref{l-ll2b}-(iii), the subgraph of $G$ induced by $S\cup T$ has a perfect matching that consist of every edge with one end-vertex in $S$ and the other one in $T$.
Hence, a maximum perfect-extendable red-blue $(S,T,X,Y)$-colouring (if it exists) is the union of

\begin{itemize}
\item [(i)] the set of edges with one end-vertex in $S$ and the other one in $T$; and
\item [(ii)] a perfect matching in $G\setminus(S \cup T)$ that contains as many edges with one end-vertex in $Z$ (and the other one in either $X\setminus S$ or $Y\setminus T$) as possible.
\end{itemize}

\noindent 
We first check in polynomial time if $G\setminus(S \cup T)$ has a perfect matching (for example, by using 
Edmonds' Blossom algorithm). If not, then $G$ has no perfect-extendable red-blue $(S,T,X,Y)$-colouring, and we stop. Otherwise, we found a perfect matching $M$ of $G\setminus (S\cup T)$, and we continue as follows. 

We colour every vertex in $X$ red and every vertex in $Y$ blue. By Lemma~\ref{l-ll2b}-(iii), every vertex in $Z$ has at most one neighbour in $X$, which belongs to $X\setminus S$, and at most one neighbour in $Y$, which belongs to $Y\setminus T$. As $Z$ is independent, we can do as follows for every $u\in Z$. If $u$ has degree~$1$ in $G$ and a neighbour $x\in X$, then $ux$ must belong to $M$, and we colour $u$ blue.  If $u$ has degree~$1$ in $G$ and a neighbour $y\in Y$, then $uy$ must belong to $M$, and we colour $u$ red. If $u$ has degree~$2$ in $G$, and a neighbour $x\in X$ and a neighbour $y\in Y$, then we colour $u$ blue if $ux\in M$ and red if $uy\in M$. This takes polynomial time. As every edge of $M$ with one end-vertex in $Z$ and the other one in either $X\setminus S$ or $Y\setminus T$ is monochromatic, we found, in polynomial time, a maximum perfect-extendable red-blue $(S,T,X,Y)$-colouring of $G$.
 \end{proof}
  
 \noindent
 The following result is proven in exactly the same way as the proof of Theorem~\ref{theo-P6} after replacing
 Lemma~\ref{l-dom} by Lemma~\ref{l-dom2};  
  Lemma~\ref{l-ll1} by Lemma~\ref{l-ll1b};
 Lemma~\ref{l-ll2} by Lemma~\ref{l-ll2b}; and
 Lemma~\ref{lem-plesnikMC} by Lemma~\ref{lem-plesnikDPM}.

\begin{theorem} \label{theo-P62}
\mdpm\ is solvable in polynomial time for $P_6$-free graphs.
\end{theorem}

\noindent
Our following result can be proven in the same way as Theorem~\ref{theo-HP2} after replacing Lemma~\ref{l-ll1} by Lemma~\ref{l-ll1b}; and Lemma~\ref{lem-plesnikMC} by Lemma~\ref{lem-plesnikDPM}.

\begin{theorem}\label{theo-HP22}
Let $H$ be a graph. If {\sc (Maximum) Disconnected Perfect Matching} is polynomial-time solvable for $H$-free graphs, then it is so for $(H+P_2)$-free graphs.
\end{theorem}

\noindent
 The following result is proven in exactly the same way as the proof of Theorem~\ref{theo-diam} after replacing 
  Lemma~\ref{l-ll1} by Lemma~\ref{l-ll1b};
 Lemma~\ref{l-ll2} by Lemma~\ref{l-ll2b}; and
 Lemma~\ref{lem-plesnikMC} by Lemma~\ref{lem-plesnikDPM}.

\begin{theorem}\label{theo-diam2}
\mdpm\ is solvable in polynomial time for graphs with diameter at most~$2$.
\end{theorem} 
 
\noindent
We now show the following result by modifying the proof of Theorem~\ref{theo-clawfree}; note that the maximum degree bound is no longer $3$ but $6$.

\begin{theorem}\label{theo-clawfree2}
\mdpm\ is \NP-hard for graphs of maximum degree~$6$ that are line graphs of triangle-free graphs.
\end{theorem}

\begin{proof}
We make the following changes in the hardness construction of Theorem~\ref{theo-clawfree}. First, we replace every vertex~$u$ of the input graph $G$ by a clique~$C_u$ of size~$6$ instead of a triangle. Then, for each edge $uv \in E(G)$, we add two edges between $C_u$ and $C_v$, such that (again) every vertex in $C_v$ has at most one neighbour outside $C_v$.
The resulting graph $G'$ is still $(K_{1,3}, \mbox{diamond})$-free but has maximum degree is~$6$.
We can now show that $G$ has an edge cut of size at least~$k$ if and only if $G'$ has a matching cut of size at least $k$, using the same arguments as in the proof of Theorem~\ref{theo-clawfree}. The proof follows from the observation that we can extend a matching cut~$M$ of $G'$ to a perfect matching by the fact that the number of vertices in a clique $C_u$ that are not incident with any edge of $M'$ is even.
\end{proof}

\noindent
We now show the following result by modifying the proof of Theorem~\ref{theo-2P3}. 

\begin{theorem}\label{theo-2P32}
\mdpm\ is \NP-hard for $2P_3$-free 
quadrangulated 
graphs of radius at most~$2$ and diameter at most~$3$.
\end{theorem}

\begin{proof}
We reduce from {\sc Exact $4$-Cover} instead of {\sc Exact $3$-Cover}. This allows us to modify the construction of the proof of Theorem~\ref{theo-2P3}, such that the triangles $K_S$ become cliques of size~$4$. This does not change the size of the matching cut. Moreover, all vertices in the clique $K_X$ still need to be matched to the cliques $K_S$ to obtain a matching cut of maximum size. However, all previously unmatched vertices, which exist only inside the cliques $K_S$, may now be matched to a vertex inside $K_S$.
\end{proof}

\noindent
We are now ready to prove Theorems~\ref{t-dichodiam3}--\ref{t-dichoH3}, which we restate below.

\medskip
\noindent
{\bf Theorem~\ref{t-dichodiam3} (restated).}
{\it For an integer~$d$,  {\sc Maximum Disconnected Perfect Matching} on graphs of diameter $d$ is 
\begin{itemize}
\item polynomial-time solvable if $d \leq 2$, and
\item \NP-hard if $d \geq 3$.
\end{itemize}}

\begin{proof}
The two results follow from Theorems~\ref{theo-diam2} and~\ref{theo-2P32}, respectively.
\end{proof}

\noindent
{\bf Theorem~\ref{t-dichorad3} (restated).}
{\it For an integer~$r$,  {\sc Maximum Disconnected Perfect Matching} on graphs of radius $r$~is 
\begin{itemize}
\item polynomial-time solvable if $r \leq 1$, and
\item \NP-hard if $r \geq 2$.
\end{itemize}}

\begin{proof}
Recall that a graph $G$ of radius~$1$ has a dominating vertex~$u$, and hence the only matching cuts are of the form $\{uv\}$, where $v$ is a vertex of degree~$1$ in $G$. Notice that such an edge $uv$ must belong to any perfect matching of $G$. Hence, we just need to check if $G$ has a perfect matching and if $G$ contains a vertex of degree~$1$.
Both can be checked in polynomial time and thus proves the first result. The second result follows from Theorem~\ref{theo-2P32}.
\end{proof}

\noindent
{\bf Theorem~\ref{t-dichoH3} (restated).}
{\it For a graph~$H$, {\sc Maximum Disconnected Perfect Matching} on $H$-free graphs is 
\begin{itemize}
\item polynomial-time solvable if $H\ssi sP_2+P_6$ for some $s\geq 0$, and
\item \NP-hard if $H\si K_{1,3}$, $2P_3$ or $H\si C_r$ for some $r\geq 3$.
\end{itemize}}

\begin{proof}
Let $H$ be a graph. First suppose that $H$ has a cycle. Recall that {\sc Disconnected Perfect Matching} is \NP-complete for $C_s$-free graphs~\cite{FLPR23} and thus for $H$-free graphs. Hence, the same holds for \mdpm. If $H$ is not a cycle, then $H$ is a forest. If $H$ contains a vertex of degree at least~$3$, then the class of $H$-free graphs contains the class of $K_{1,3}$-free graphs, which contains the class of line graphs, so we apply Theorem~\ref{theo-clawfree2}. 
Otherwise $H$ is a linear forest. If $H\ssi sP_2+P_6$ for some $s\geq 0$, then we apply 
Theorem~\ref{theo-P62} in combination with $s$ applications of Theorem~\ref{theo-HP22}. Else $H$ contains an induced $2P_3$ and we apply Theorem~\ref{theo-2P32}. 
\end{proof}

\section{Conclusions}\label{s-con}

We considered the optimization version {\sc Maximum Matching Cut} of the classical {\sc Matching Cut} problem after first observing that the {\sc Perfect Matching Cut} problem is a special case of the former problem. We generalized known algorithms for graphs of diameter at most~$2$ and $P_6$-free graphs from {\sc Matching Cut} and {\sc Perfect Matching Cut} to {\sc Maximum Matching Cut}. We also showed that the latter problem is computationally harder (assuming $\PP\neq \NP)$ than {\sc Matching Cut} and {\sc Perfect Matching Cut} for various graph classes. Our results led to three new dichotomy results, namely a computational complexity classification of {\sc Maximum Matching Cut} for $H$-free graphs, and complexity classifications for graphs of bounded diameter and bounded radius. Classification for $H$-free graphs are still unsettled for {\sc Matching Cut} and {\sc Perfect Matching Cut}, as can be observed from Theorems~\ref{t-main1} and~\ref{t-main2}. We also pose the following open problem, which is the missing case from Theorem~\ref{t-diameter2}.

\begin{open}\label{o-1}
Determine the complexity of {\sc Perfect Matching Cut} for graphs of diameter~$3$.
\end{open}

\noindent
To prove the dichotomies for {\sc Maximum Matching Cut}, we showed that {\sc Maximum Matching Cut} is \NP-hard for $2P_3$-free quadrangulated graphs of diameter~$3$ and radius~$2$, whereas {\sc Matching Cut} is known to be polynomial-time solvable for quadrangulated graphs~\cite{Mo89}. We recall an open problem of Le and Telle~\cite{LT22} who asked, after proving polynomial-time solvability for chordal graphs, the following question (a graph is {\it $k$-chordal} for some $k\geq 3$ if it is $(C_{k+1},C_{k+2},\ldots)$-free, so $3$-chordal graphs are the chordal graphs).

\begin{open}[\cite{LT22}]
Determine the complexity of {\sc Perfect Matching Cut} for quadrangulated graphs, or more general, $k$-chordal graphs for $k\geq 4$.
\end{open}

\noindent
We also showed how our proofs could be adapted to hold for {\sc Maximum Disconnected Perfect Matching}, the optimization version of {\sc Disconnected Perfect Matching}. This led to exactly the same dichotomies for the former problem as for {\sc Maximum Matching Cut} for bounded diameter, bounded radius and $H$-free graphs. Moreover, it implied new results for {\sc Disconnected Perfect Matching} as well, including a polynomial-time algorithm for $P_6$-free graphs.
The complexity classification of {\sc Disconnected Perfect Matching} for $H$-free graphs is still not complete (see Theorem~\ref{t-main3}). We also pose the following open problem, which is the missing case from Theorem~\ref{t-diameter3}.

\begin{open}\label{o-3}
Determine the complexity of {\sc Disconnected Perfect Matching} for graphs of radius~$2$.
\end{open}

\noindent
Our final open problem is related to $H$-free graphs.

\begin{open}
For every graph $H$, is {\sc Disconnected Perfect Matching} polynomial-time solvable for $(H+P_3)$-free graphs if it is polynomial-time solvable for $H$-free graphs?
\end{open}

\noindent
We now know from Theorem~\ref{t-main3} that the above result holds for $P_2$, while for {\sc Matching Cut} we have this result for $P_3$~\cite{LPR22} and for {\sc Perfect Matching Cut} even for $P_4$~\cite{LPR23a}.

\medskip
\noindent
{\it Acknowledgments.} We thank Van Bang Le for pointing out that Theorem~\ref{t-dichoH3} implies a resolution of the open problem of~\cite{BP} for $P_6$-free graphs.

\end{document}